\documentclass[10pt,english]{smfart}%{smfart}amsart
\usepackage{amsfonts,amsmath,enumerate,latexsym,verbatim,amscd,mathrsfs,color,array,dsfont,appendix}%,newclude}
\newtheorem{theorem}{Theorem}[section]

\newtheorem{proposition}[theorem]{Proposition}
\newtheorem{lemma}[theorem]{Lemma}
\newtheorem{definition}[theorem]{Definition}
\newtheorem{notation}[theorem]{Notation}
\newtheorem{corollary}[theorem]{Corollary}

\renewcommand{\theequation}{\thesection.\arabic{equation}} %thesubsection

\usepackage[colorlinks=true]{hyperref}
\hypersetup{urlcolor=blue,linkcolor=black,citecolor=black,colorlinks=true}

\title[Smoluchowski's equation]{Smoluchowski's equation: rate of convergence of the Marcus-Lushnikov process}

\author{Eduardo Cepeda and Nicolas Fournier}
\address{Laboratoire d'Analyse et de Math\'ematiques Appliqu\'ees, UMR 8050. Universit\'e Paris-Est. 61, avenue du G\'en\'eral de Gaulle, 94010 Cr\'eteil C\'edex}
\email{eduardo.cepeda-chiluisa@etu.univ-paris12.fr, nicolas.fournier@univ-paris12.fr}

\begin{document}

\begin{abstract} 
We derive a satisfying rate of convergence of the Marcus-Lushnikov process toward the solution to Smoluchowski's coagulation equation. Our result applies to a class of homogeneous-like coagulation kernels with homogeneity degree ranging in $(-\infty,1]$. It relies on the use of a Wasserstein-type distance, which has shown to be particularly well-adapted to coalescence phenomena. It was introduced and used in preceeding works Fournier and Lauren\c{c}ot (2006) and Fournier and L\"ocherbach (2009). \medskip 

\noindent
\textbf{Mathematics Subject Classification (2000)}: 60H30, 45K05.\medskip

\noindent
\textbf{\textit{Keywords}}: Smoluchowski's coagulation equation, Marcus-Lushnikov process, Interacting stochastic particle systems.
\end{abstract}
\maketitle
%\tableofcontents
To appear in \textit{``Stochastic Procceses and their Applications"} (accepted on march 2011).
\nocite{Aldous} \nocite{Rate}  \nocite{Eibeck_Wagner} \nocite{Sto-Coal}  \nocite{Conv_ML} \nocite{Well-Pdnss} \nocite{Sto-Coal2} \nocite{Jacod} \nocite{Jeon} \nocite{Jourdain} \nocite{Kolokoltsov} \nocite{Lushnikov} \nocite{Marcus} \nocite{Ziemer} 
\section{Introduction}\label{Introduction}
 
We are interested in coalescence which is a widespread phenomenon: it arises in physics, chemistry, astrophysics, biology and mathematics.\medskip

We consider a possibly infinite system of particles, each particle being fully identified by its mass ranging in the set of positive real numbers. The only mechanism taken into account is the coalescence of two particles with masses $x$ and $y$ into a single one with mass $x+y$ at some given rate (the ``coagulation kernel'') $K(x,y)=K(y,x)\geq 0$.\medskip

\begin{itemize}
\item We can consider a system of microscopic particles and the following system of differential equations for the concentrations  $\mu_t(x)$ of particles of mass $x = 1, 2, 3,...$ at time $t\in[0,+\infty)$:
\begin{equation} \label{Intro:SmoEq_discrete}
\partial_t \mu_t(x) = \dfrac{1}{2} \sum_{y=1}^{x-1} K(y,x-y) \mu_t(y) \mu_t(x-y) - \mu_t(x)\sum_{y=1}^{+\infty}  K(x,y) \mu_t(y).
\end{equation}
The first sum in (\ref{Intro:SmoEq_discrete}) on the right corresponds to coagulation of smaller particles to produce one of mass $x$, whereas the second sum corresponds to removal of particles of mass $x$ as they in turn coagulate to produce larger particles.\medskip

Analogous integro-differential equations allow us to consider a continuum of masses $x$. In this case the system can also be described by the concentration $\mu_t(x)$ of particles of mass $x\in(0,+\infty)$ at time $t\in[0,+\infty)$. Then $\mu_t(x)$ solves a nonlinear equation:
\begin{equation} \label{Intro:SmoEq}
\partial_t \mu_t(x) = \dfrac{1}{2} \int_0^x K(y,x-y) \mu_t(y) \mu_t(x-y) dy - \mu_t(x)\int_0^{+\infty} K(x,y) \mu_t(y)dy.
\end{equation}
Equation (\ref{Intro:SmoEq}) is known as the continuous Smoluchowski coagulation equation and (\ref{Intro:SmoEq_discrete}) is its discrete version.\medskip

\item When the particles are macroscopic and when the rate of coagulation is not infinitesimal, the frame of study of the dynamics of such a system is stochastic. When the initial state consists of a finite number of macroscopic particles, the stochastic coalescent obviously exists (see \cite{Aldous}) and it is known as the Marcus-Lushnikov process.\medskip
\end{itemize}

In preceding works several results have been obtained on the existence and uniqueness of weak solutions to Smoluchowski's coagulation equation. The general framework was formulated in \cite{Norris} who obtained some remarkable well-posedness results. In \cite{Well-Pdnss}, homogeneous-like kernels are considered and it has been seen that the well-posedness holds in the class of measures having a finite moment of order the degree of homogeneity of the coagulation kernel.\medskip

Aldous \cite{Aldous} presents the Marcus-Lushnikov process as an approximation for the solution of Smoluchowski's equation (see \cite{Marcus, Lushnikov} for further information). Since then some results on convergence have been obtained in \cite{Norris} and \cite{Jeon}, see also \cite{Conv_ML}. A class of stochastic algorithms in which the number of particles remains constant in time was introduced in \cite{Eibeck_Wagner} and has been extended to the discrete coagulation-fragmentation case in \cite{Jourdain}.\medskip

We investigate the rate of convergence of the Marcus-Lushnikov process to the solution of the Smoluchowski coagulation equation as the number of particles tends to infinity. This problem is interesting because on the one hand it has a physical meaning: the Smoluchowski equation is often derived by passing to the limit in the Marcus-Lushnikov process, and on the other hand from a numerical point of view: this stochastic process can be simulated exactly. Thus it seems natural to use it in order to approximate the solution to Smoluchowski's coagulation equation.\medskip

Our study is based on the use of a specific Wasserstein-type distance $d_{\lambda}$ between the solution to Smoluchowski's equation and its stochastic approximation. This distance depends on the homogeneity parameter $\lambda$ of the coagulation kernel. This specific distance has been introduced in \cite{Well-Pdnss} to prove some results on the well-posedness of the Smoluchowski coagulation equation and in \cite{Sto-Coal, Sto-Coal2} to study the stochastic coalescent. The result of the present work applies to a family of homogeneous-like coagulation kernels. These kernels are of particular importance in applications see Table 1 in \cite{Aldous} or the list provided in \cite{Well-Pdnss}. \medskip

We point out that since we are using a finite particle system to approximate the evolution in time of the solution to the Smoluchowski equation which describes an infinite particle system, it is necessary to dispose of a mechanism to construct an initial condition for the Marcus-Lushnikov process from a general measure-valued initial condition of Smoluchowski's equation. This initial condition needs to satisfy, on the one hand, a convergence condition to assure the convergence of the stochastic process to the solution to Smoluchowski's equation for all time $t$ as the number of particles grows (the usual condition of weak convergence is replaced by convergence in the sense of the distance we use), and on the other hand it must obey a rate of convergence in order to control the overall rate of convergence of such an approximation.\medskip

Very roughly, we consider a homogeneous-like coagulation kernel with degree of homogeneity $\lambda \in (-\infty,1]\setminus \{0\}$ (including $K(x,y) = (x+y)^{\lambda}$). For $(\mu_t)_{t\geq 0}$ the solution to the corresponding Smoluchowski's equation and for $(\mu^n_t)_{t\geq 0}$ the corresponding Marcus-Lushnikov process, we prove that
\begin{equation*}
\sup_{t\in[0,T]} \mathbb E\left[d_{\lambda}(\mu^n_t,\mu_t)  \right]  \leq \frac{C_{T}}{\sqrt{n}},
\end{equation*}
as soon as $\mu_0$ satisfies some technical conditions and for a good choice of the initial state of the Marcus-Lushnikov process $\mu_0^n$ of the form $\frac{1}{n} \sum_{k=1}^N \delta_{x_{k}} $. We can make the following remarks.\medskip

\textit{1.} Recalling the \textit{Central Limit Theorem} (CLT), this rate of convergence seems to be optimal, since the convergence of $\mu^n_t$ to $\mu_t$ is a generalized Law of Large Numbers.\medskip

\textit{2.} In \cite{Well-Pdnss} it has been seen that only one moment is demanded to show the well-posedness for the Smoluchowski equation. In the present work, we need to demand more moments, but we believe that it is very difficult to avoid such conditions.\medskip

\textit{3.} The only works giving an explicit result on the rate of convergence of the Marcus-Lushnikov process toward the solution to Smoluchowski's coagulation equation, known by us, are:
\begin{itemize}
\item[-] Norris \cite{Norris}, who gives an estimate using a ``Large Deviations" approach for the discrete case ($\textrm{supp}(\mu_0)\subset \mathbb N$).
\item[-] Deaconu, Fournier and Tanr\'e \cite{Rate}, where a CLT-type result is shown for the discrete case and for a bounded coagulation kernel $K$, furthermore in this work a different particle system is used.
\item[-] Kolokoltsov \cite{Kolokoltsov}, who uses analytic methods of the theory of semigroups applied to the
Markov infinitesimal generator. He also uses a different distance to ours, namely the author uses the topology of the dual to the weighted spaces of continuously differentiable functions or certain weighted Sobolev spaces. He then gives a CLT result for the discrete case with a coagulation kernel satisfying $K(x,y)\leq c(1+\sqrt{x})(1+\sqrt{y})$ and for the continuous case when $K$ is two times differentiable with all its derivatives bounded. Unfortunately the case $K(x,y) = (x+y)^{\lambda}$ is excluded for any value of $\lambda \in (-\infty,1]\setminus\{0\}$.
\end{itemize}
Our work thus gives the first result on the rate of convergence covering the continuous case for some homogeneous kernels. \medskip

For the case $\lambda<0$ we follow the ideas found in \cite{Well-Pdnss}, but for the case $\lambda \in (0,1]$ the proof is much more difficult and the calculations are faced in a completely different way. Namely we use the It\^o formula for an approximation of the absolute value function and handle very delicately the resulting terms.\medskip

The paper is organized as follows: in Section \ref{Notation_Assumptions_and_Definitions} we give the notation and definitions we use in this document, in Section \ref{Results} we state our main result. The proof is developed in Sections \ref{Negative_Case}, \ref{Positive_Case} and \ref{Special_Case}. We give also a method to construct an initial condition for the Marcus-Lushnikov process in Section \ref{Choice_Mu0} and we conclude the document giving some technical details which are useful all along the paper in Appendix \ref{Preliminaries}. 

\section{Notation, Assumptions and Definitions}\label{Notation_Assumptions_and_Definitions}
\setcounter{equation}{0}

In this section we present our assumptions, give the definition of weak solutions to Smoluchowski's coagulation equation and then we recall the dynamics of the Marcus-Lushnikov process.

\begin{notation}\label{notation}
We denote by $\mathcal M^+$ the space of non-negative Radon measures on $(0,+\infty)$. For a measure $\mu$ and a function $\phi$, we set $\left\langle\mu(dx)\,,\,\phi(x) \right\rangle = \int_0^{+\infty}\phi(x)\mu(dx)$. We also define the operator $A$ for all measurable functions $\phi:(0,+\infty) \rightarrow \mathbb R$, by
\begin{equation} \label{Intro:Op-A}
(A\phi)(x,y) = \phi(x+y)-\phi(x) -\phi(y)\,\,\,\, \forall\, (x,y) \in (0,+\infty)^2.
\end{equation}
Finally, we will use the notation $x\wedge y = \min\{ x,y\}$ and $x \vee y = \max\{ x,y\}$ for $(x,y) \in (0,+\infty)^2$.
\end{notation}

We consider a coagulation kernel $K:(0,+\infty)\times (0,+\infty) \rightarrow [0,+\infty)$, symmetric i.e. $K(x,y)=K(y,x)$ for $(x,y)\in(0,+\infty)^2$. We further assume it belongs to $W^{1,\infty} \left((\varepsilon,1/\varepsilon)^2\right)$ for every $\varepsilon \in (0,1)$ and one of the following conditions  $\forall\,(x,y)\in(0,+\infty)^2$:
\begin{eqnarray}
\hspace{1cm} \lambda\in(-\infty,0), &  K(x,y)\leq \kappa_0\,(x+y)^{\lambda} \textrm{ and } \left(x^{\lambda}+y^{\lambda}\right)\left|\partial_x K(x,y)\right|\leq \kappa_1 x^{\lambda-1}y^{\lambda},\label{case1}\\[3mm]
\lambda\in(0,1],\hspace{0.5cm}  &  K(x,y)\leq \kappa_0\,(x+y)^{\lambda} \textrm{ and } \left(x^{\lambda} \wedge y^{\lambda}\right)\left|\partial_x K(x,y)\right|\leq \kappa_1 x^{\lambda-1}y^{\lambda},\label{case2}\\[3mm]
 \lambda\in(0,1],\hspace{0.5cm} &  K(x,y)\leq \kappa_0\,(x\wedge y)^{\lambda} \textrm{ and } \left(x^{\lambda} \wedge y^{\lambda}\right)\left|\partial_x K(x,y)\right|\leq \kappa_1 x^{\lambda-1}y^{\lambda},\label{special_case}
\end{eqnarray} 
for some positive constants $\kappa_0$ and $\kappa_1$. We refer to \cite{Well-Pdnss} for a list of physical kernels satisfying conditions (\ref{case1}) and (\ref{case2}). Remark that for any $\lambda\in(-\infty,1] \setminus  \{0\}$, $K(x,y) = (x+y)^{\lambda}$ satisfies (\ref{case1}) or (\ref{case2}).
\begin{definition}\label{Intro:MetricDef}
Consider $\lambda \in (-\infty,1]\setminus\{0\}$. For $\mu \in \mathcal M^+$, we set:
\begin{equation}\label{Intro:MomentDef}
 M_{\lambda}(\mu) = \int_0^{+\infty} x^{\lambda} \mu(dx) \hspace{7.5mm}\textrm{and}\hspace{7.5mm}
 \mathcal M^{+}_{\lambda} = \{ \nu \in \mathcal M^{+}: M_{\lambda}(\nu) < +\infty  \}.
\end{equation}
For $\mu \in \mathcal M^+$, we set, for $x\in (0,+\infty)$:
\begin{equation}\label{Intro:MetricEq1} 
F^{\mu}(x) = \int_0^{+\infty} \mathds 1_{(x,+\infty)}(y)\, \mu(dy) \hspace{3mm} \textrm{and } \hspace{3mm} G^{\mu}(x) = \int_0^{+\infty} \mathds 1_{(0,x]}(y)\, \mu(dy). 
\end{equation} 
We define the distance on $\mathcal M^+_{\lambda}$ as
\begin{equation}\label{Intro:MetricEq3}
d_{\lambda}(\mu,\tilde{\mu}) = \int_0^{+\infty} x^{\lambda-1} |E(x)| dx,
\end{equation}
where $ E(x) = G^{\mu}(x) - G^{\tilde{\mu} }(x)$ if $\lambda\in(-\infty,0)$ and $ E(x) = F^{\mu}(x) - F^{\tilde{\mu}}(x)$ if  $\lambda\in(0,1]$.
\end{definition}

We remark that $d_{\lambda}$ is well-defined on $\mathcal M^+_{\lambda}$. Indeed we have $d_{\lambda}(\mu,\tilde{\mu}) \leq \frac{1}{|\lambda|} M_{\lambda}\left(\mu + \tilde{\mu}\right)$ for $\lambda \in (-\infty,1]\setminus\{0\}$. See \cite{Dist-Coag} for a deeper study of this distance in the discrete and continuous cases.\medskip

We excluded the case $\lambda=0$ for two reasons. First, $d_0$ is not well-defined on $\mathcal M^+_0$. Next, when trying to extend our study to this case, we are not able to obtain a better result than those of Kolokoltsov \cite{Kolokoltsov}.

\begin{definition}%[Space of test functions $\mathcal H_{\lambda}$]
For $\lambda\in(-\infty,1]\setminus \{0\}$ we introduce the spaces of test functions needed to define weak solutions:
\[
\begin{array}{ll}
if\,\,\,\lambda \in(-\infty,0): & \mathcal H_{\lambda} = \left\lbrace \phi: (0,+\infty) \rightarrow \mathbb R \textrm{ such that }\sup_{x>0} x^{-\lambda}|\phi(x)|<+\infty \right\rbrace, \\[3mm]
if\,\,\,\lambda \in(0,1]: & \mathcal H_{\lambda} = \left\lbrace \phi: (0,+\infty) \rightarrow \mathbb R \textrm{ such that }\sup_{x>0} (1 + x)^{-\lambda}|\phi(x)|<+\infty \right\rbrace,\\[3mm]
if\,\,\,\lambda \in(0,1]: & \mathcal H^e_{\lambda} = \left\lbrace \phi: (0,+\infty) \rightarrow \mathbb R \textrm{ such that }\sup_{x>0} x^{-\lambda}|\phi(x)|<+\infty \right\rbrace.
\end{array}
\]
It is necessary to introduce the space $\mathcal H^e_{\lambda}$ to study the case (\ref{special_case}).
\end{definition}

\subsection{The Smoluchowki coagulation equation}\label{Intro_SmoEq}

The weak formulation of the Smoluchowski coagulation equation is given by
\begin{equation} \label{Intro:SmoEq_Weak}
\dfrac{d}{dt}\langle \mu_t(dx), \phi(x) \rangle = \dfrac{1}{2} \langle \mu_t(dx)  \mu_t(dy), (A\phi)(x,y)K(x,y)\rangle , 
\end{equation}
see Notation \ref{notation}. This is a general formulation and it embraces the two previous equations : if $\mu_0$ is discrete (i.e. $\textrm{supp}(\mu_0)\subset \mathbb N$), then this corresponds to the ``discrete coagulation equation'' (\ref{Intro:SmoEq_discrete}), while when $\mu_0$ is continuous (i.e. $\mu_0(dx) = \mu_0(x)dx$), this corresponds to the ``continuous coagulation equation'' (\ref{Intro:SmoEq}). Formulation (\ref{Intro:SmoEq_Weak}) is standard, see \cite{Norris}.\medskip

\begin{definition}\label{Def_SmoEq} Let $\lambda\in(-\infty,1]\setminus \{0\}$, a coagulation kernel $K$ satisfying either (\ref{case1}), (\ref{case2}) or (\ref{special_case}), and $\mu^{in} \in \mathcal M^{+}_{\lambda}$. We will then say that $(\mu_t)_{t\geq 0} \subset \mathcal M^{+}$ is a $(\mu^{in},K,\lambda)$-weak solution to Smoluchowski's equation if the following conditions are verified: 
\begin{enumerate}[(i)]
\item $\mu_0 = \mu^{in}$,
\item the application $t \longmapsto \langle \mu_t(dx), \phi(x) \rangle$
is differentiable on $[0,+\infty)$ and satisfies (\ref{Intro:SmoEq_Weak}) for each $\phi \in \mathcal H_{\lambda}$ (cases (\ref{case1}) and (\ref{case2})) or for each $\phi \in \mathcal H^e_{\lambda}$ (case (\ref{special_case})),
\item for all $T\in[0,+\infty)$ 
\begin{equation}\label{Intro:SupBounded}
\sup_{s\in[0,T]} M_{\alpha}(\mu_s) < +\infty,
\end{equation}
for $\alpha = \lambda$ (cases (\ref{case1}) and (\ref{special_case})) or for $\alpha = 0,\,2\lambda$ (case (\ref{case2})).
\end{enumerate}
\end{definition}
We demand more finite moments of $\mu_0$ than in \cite{Well-Pdnss} to assure the convergence of the Marcus-Lushnikov process. According to the hypothesis on the kernel (\ref{case1}), (\ref{case2}) or (\ref{special_case}) together with (\ref{Intro:SupBounded}) and  Lemma \ref{Intro:lemmaWPdnss}, the integrals in the weak formulation (\ref{Intro:SmoEq_Weak}) are absolutely convergent and bounded with respect to $t\in[0,T]$ for every $T$. \medskip

Under (\ref{case1}) or (\ref{special_case}), the existence and uniqueness of such weak solutions have been established in \cite{Well-Pdnss} for any $\mu^{in}\in \mathcal M^+_{\lambda}$. Under (\ref{case2}), the existence and uniqueness of  weak solutions satisfying (\ref{Intro:SupBounded}) with $\alpha = \lambda$ have also been checked in \cite{Well-Pdnss} for any $\mu^{in}\in \mathcal M^+_{\lambda}$. Using furthermore Proposition \ref{Prop:bounded_moments}, we immediately deduce the existence and uniqueness of weak solutions under (\ref{case2}), in the sense of Definition \ref{Def_SmoEq}, for any $\mu^{in}\in \mathcal M^+_{0} \cap \mathcal M^+_{2\lambda}$.

\subsection{The Marcus-Lushnikov process}\label{Intro_StoMLp}

The Marcus-Lushnikov process describes the stochastic Markov evolution of a finite particle system of coalescing particles. We consider a coagulation kernel $K$  and a finite particle system initially consisting of $N \geq 2$ particles of masses $x_1,  \cdots, \, x_N \in (0,+\infty)$.  We assume that the system evolves according to the following dynamics: each pair of particles (of masses $x$ and $y$) coalesce (i.e. disappears and forms a new particle of mass $x+y$) with a rate proportional to $K(x,y)$. \medskip

Let $n\in\mathbb N$ and we assign to all particles the weight $1/n$. We define now rigorously the Marcus-Lushnikov process to be used.
\begin{definition}
We consider a coagulation kernel $K$, $n\in \mathbb N$ and an initial state $\mu^n_0 = \dfrac{1}{n} \sum_{i=1}^N \delta_{x_i}$, with $x_1,  \cdots, \, x_N \in (0,+\infty)$. \medskip

The Marcus-Lushnikov process $(\mu^n_t)_{t\geq 0}$ associated with $(n,K,\mu^n_0)$ is a Markov $\mathcal M^+$-valued c\`adl\`ag process satisfying:
\begin{enumerate}[(i)]
\item $(\mu^n_t)_{t\geq 0}$ takes its values in $\left\lbrace \dfrac{1}{n} \sum_{i=1}^k \delta_{y_i}; k\leq N,\,y_i>0 \right\rbrace $.
\item Its infinitesimal generator is given, for all mesurable functions $\Psi:\mathcal M^+ \rightarrow \mathbb R$ and all states $ \mu = \dfrac{1}{n} \sum_{i=1}^k \delta_{y_i}$ by
\begin{equation*}
L \Psi(\mu) = \sum_{1\leq i<j\leq k} \left\{\Psi\left[\mu + n^{-1}\left(\delta_{y_i+y_j}-\delta_{y_i} -\delta_{y_j}\right)\right] - \Psi[\mu] \right\} \dfrac{K(y_i,y_j)}{n}.
\end{equation*}
\end{enumerate}
\end{definition}
This process is known to be well-defined and unique, see \cite{Aldous,Norris}. We will use the following classical representation of the Marcus-Lushnikov process (see e.g. \cite{Sto-Coal,Sto-Coal2}): there is a Poisson measure $J(dt,d(i,j),dz)$ on $[0,+\infty) \times \{(i,j)\in\mathbb N^2, i<j\}\times [0,+\infty)$ with intensity measure $dt \left[\sum_{k < l} \delta_{(k,l)}\left(d(i,j)\right)\right] dz$, such that for any measurable function $\phi: (0,+\infty) \rightarrow \mathbb R$
\begin{eqnarray}\label{Dev:Eq-mu_n-1}
\langle\mu^n_t(dx), \phi(x)\rangle & = &\langle\mu^n_{0}(dx), \phi(x)\rangle\, +\,\int_0^t \int_{i<j} \int_0^{+\infty} \dfrac{1}{n}\left[ \phi\left(X^i_{s-}+ X^j_{s-}\right) - \phi\left(X^i_{s-}\right) - \phi\left(X^j_{s-}\right)\right]   \nonumber \\
&  & \hspace{4.5cm}\mathds 1_{\left\lbrace z \leq \frac{K\left(X^i_{s-},X^j_{s-}\right)}{n} \right\rbrace }\,\mathds 1_{\{j\leq N(s-)\}}J(ds,d(i,j),dz),
\end{eqnarray}
where $\mu^n_t = \frac{1}{n}\sum^{N(t)}_{k=1} \delta_{X^k_t}$, $N(t)$ being the (non-increasing) number of particles at time $t$.\medskip

This can be written using the compensated Poisson measure related to $J$:
\begin{eqnarray}\label{Dev:Eq-mu_n-2}
\langle\mu^n_t(dx), \phi(x)\rangle & = &\langle\mu^n_0(dx), \phi(x)\rangle + \, \dfrac{1}{2} \int_0^t \langle \mu^n_s(dx) \mu^n_s (dy), \left(A\phi\right)(x,y) K(x,y)\rangle ds\nonumber\\
& & -\, \dfrac{1}{2n} \int_0^t \langle \mu^n_s(dx), \left(A\phi\right)(x,x) K(x,x) \rangle ds\\
& & +\, \int_0^t \int_{i<j} \int_0^{+\infty} \dfrac{1}{n}\, \left(A\phi\right)\left(X^i_{s-},X^j_{s-}\right)\, \mathds 1_{\left\lbrace z \leq \frac{K\left(X^i_{s-},X^j_{s-}\right)}{n} \right\rbrace }\mathds 1_{\{j\leq N(s-)\}}\nonumber\\
& & \hspace{8cm}\tilde{J}(ds,d(i,j),dz), \nonumber
\end{eqnarray}
where the operator $A$ is defined in (\ref{Intro:Op-A}). The third term on the right-hand side is issued from the impossibility of coalescence of a particle with itself. 

\section{Results}\label{Results}
\setcounter{equation}{0}

We state in this section our main result. We also state as a proposition the construction of a sequence of initial conditions for the Marcus-Lushnikov processes and finally comment on our results.

\begin{theorem}\label{theorem}
We consider $\lambda \in (-\infty,1]\setminus\{0\}$ and a coagulation kernel $K$ satisfying either (\ref{case1}), (\ref{case2}) or (\ref{special_case}). Let $\mu_0 \in \mathcal M^{+}$ and $(\mu_t)_{t\geq 0}$ the $(\mu_{0},K,\lambda)$-weak solution to Smoluchowski's equation. Let $\mu^n_0$ be deterministic and of the form $\frac{1}{n}\sum_{i=1}^N \delta_{x_i}$ and denote by $(\mu^n_t)_{t\geq 0}$ the associated $(n,K,\mu^n_0)$-Marcus-Lushnikov process. Let $\varepsilon>0$.
\begin{itemize}
\item[$\bullet$] Assume (\ref{case1}) or (\ref{special_case}) and that $\mu_0$ belongs to $\mathcal M^+_{\lambda} \cap \mathcal M^+_{2\lambda + \tilde{\varepsilon}}$, where $\tilde{\varepsilon} = sgn(\lambda)\times \varepsilon$. Then for any $T>0$,
\begin{eqnarray*}
\hspace{5mm}\mathbb E\left[\sup_{t\in[0,T]} d_{\lambda}(\mu^n_t,\mu_t)  \right]  &\leq & \left[d_{\lambda}(\mu^n_0,\mu_0) + \frac{(1+T)C_{\lambda,\varepsilon}}{\sqrt{n}} \bigg(M_{\lambda}(\mu^n_0) + M_{2\lambda+\tilde{\varepsilon}}(\mu^n_0)\bigg)\right] \\
& &\times \exp\left[T C_{\lambda,\varepsilon}  M_{\lambda}(\mu^n_0+\mu_0) \right],\nonumber
\end{eqnarray*}
where $C_{\lambda,\varepsilon}$ is a positive constant depending only on $\lambda$, $\varepsilon$ and $\kappa_0$, and $\kappa_1$.
\item[$\bullet$] Assume (\ref{case2}) and that $\mu_0\in\mathcal M^+_{0} \cap \mathcal M^+_{\gamma+\varepsilon}$ where $\gamma = \max\{2\lambda, 4\lambda-1\}$. Then for any $T>0$,
\begin{eqnarray*}
\hspace{5mm}\sup_{t\in[0,T]}\mathbb E\left[ d_{\lambda}(\mu^n_t,\mu_t)  \right]  &\leq & \Bigg[d_{\lambda}(\mu^n_0,\mu_0) +  \frac{(1+T)C_{\lambda,\varepsilon}}{\sqrt{n}} \bigg( 1 + \left[M_{0}(\mu^n_0 + \mu_0)\right]^2  \\
& &  + \left[M_{\gamma+\varepsilon}(\mu^n_0 + \mu_0)\right]^2 \bigg)\Bigg]\times \exp\left[T C_{\lambda,\varepsilon} M_{\lambda}(\mu^n_0+\mu_0) \right],\nonumber
\end{eqnarray*}
where $C_{\lambda,\varepsilon}$ is a positive constant depending only on $\lambda$, $\varepsilon$, $\kappa_0$ and $\kappa_1$.
\end{itemize}
\end{theorem}

Now we present the proposition giving a $d_{\lambda}$-approximation of the initial condition.

\begin{proposition}\label{Results:Prop}
Let $\lambda \in (-\infty,1] \setminus \{0\}$, $n\in \mathbb N$ and $\mu_0$ a non negative Radon measure on $(0,+\infty)$ such that  $\mu_0\in \mathcal M^+_{\lambda} \cap \mathcal M^+_{2\lambda}$. The measure $\mu_0$ is supposed to be either atomless or discrete ($\rm{supp}(\mu_0)\subset \mathbb N$). 
Then, there exists a positive measure $\mu^n_0$ of the form $\frac{1}{n}\sum_{i=1}^{N_n} \delta_{x_i}$ such that:
\begin{equation*}
d_ {\lambda}(\mu^n_0,\mu_0) \leq \frac{C_{\lambda}}{\sqrt{n}},
\end{equation*}
where the constant $C_{\lambda}$ depends only on $\lambda$ and  $M_{2\lambda}(\mu_0)$. We also have
\begin{equation*}
M_{\alpha}(\mu^n_0) \leq M_{\alpha}(\mu_0),
\end{equation*}
for all $\alpha\leq0$ if $\lambda \in (-\infty,0)$ and for all $\alpha\geq 0$ if $\lambda \in (0,1]$. Furthermore, if $M_0(\mu_0)< +\infty$, then 
\begin{equation*}
N_n \leq n\, M_0(\mu_0).
\end{equation*}
\end{proposition}
The estimate of the parameter $N_n$ (initial number of particles) may be useful to study the numerical cost of the simulation.\medskip

Gathering Theorem \ref{theorem} and Proposition \ref{Results:Prop}, we deduce the following statement.
\begin{corollary}\label{corollary}
We consider $\lambda \in (-\infty,1]\setminus\{0\}$, $\varepsilon>0$ and a coagulation kernel $K$ satisfying either (\ref{case1}), (\ref{case2}) or (\ref{special_case}). Let $\mu_0 \in \mathcal M^{+}$  be either atomless or discrete ($\textrm{supp}(\mu_0)\subset \mathbb N$), and $(\mu_t)_{t\in[0,+\infty)}$ the $(\mu_{0},K,\lambda)$-weak solution to Smoluchowski's equation. Then it is possible to build a family of initial conditions $\mu_0^n=\frac 1 n \sum_{k=1}^{N_n} \delta_{x_i}$ such that, for $(\mu^n_t)_{t\geq 0}$ the corresponding $(n,K,\mu^n_0)$-Marcus-Lushnikov process,
\begin{itemize}
\item[$\bullet$] under (\ref{case1}) or (\ref{special_case}), if $\mu_0$ belongs to $\mathcal M^+_{\lambda} \cap \mathcal M^+_{2\lambda + \tilde{\varepsilon}}$, where $\tilde{\varepsilon} = sgn(\lambda)\times \varepsilon$, then for any $T>0$,
\begin{eqnarray}
\hspace{5mm}\mathbb E\left[\sup_{t\in[0,T]} d_{\lambda}(\mu^n_t,\mu_t)  \right]  &\leq & \frac{C_{T}}{\sqrt{n}}, \nonumber
\end{eqnarray}
where $C_{T}$ is a positive constant depending only on $T$, $\lambda$, $\varepsilon$, $\kappa_0$, $\kappa_1$ and $\mu_0$;
\item[$\bullet$] under (\ref{case2}), if $\mu_0\in\mathcal M^+_{0} \cap \mathcal M^+_{\gamma+\varepsilon}$ where $\gamma = \max\{2\lambda, 4\lambda-1\}$, then for any $T>0$,
\begin{eqnarray}
\hspace{5mm}\sup_{t\in[0,T]}\mathbb E\left[ d_{\lambda}(\mu^n_t,\mu_t)  \right]  &\leq & \frac{C_{T}}{\sqrt{n}}, \nonumber
\end{eqnarray}
where $C_{T}$ is a positive constant depending only on $T$, $\lambda$, $\varepsilon$, $\kappa_0$, $\kappa_1$ and $\mu_0$. 
\end{itemize}
\end{corollary}
This last statement is quite satisfying since it provides a rate of convergence in $\frac{1}{\sqrt{n}}$ and it applies to a large class of homogeneous kernels presenting singularities for small or large masses. We probably require more finite moments than really needed but this does not seem to be a real problem for applications.\medskip

We have followed the ideas found in \cite{Well-Pdnss} to prove the case (\ref{case1}) and the special case (\ref{special_case}) of Theorem \ref{theorem}. The case (\ref{case2}) is much more subtle and difficult. For this case we have applied the It\^o formula and manipulated very carefully each term.  By the moment it is not possible to put the ``$\sup$" into the expectation since it is very important to use the sign of the terms and to take advantage of some cancelations.\medskip

Proposition \ref{Results:Prop} presents the proof of the existence of a $d_{\lambda}$-approximation of a general non-negative measure $\mu_0$ (we consider measures $\mu_0$ which are interesting for the Smoluchowski's equation) by a discrete measure $\mu^n_0$ (a finite sum of Dirac's deltas) as a construction procedure. This construction is very useful from a numerical point of view since it gives a measure that will be set as the initial state for the Marcus-Lushnikov process. 

\section{Negative Case}\label{Negative_Case}
\setcounter{equation}{0}
In the whole section, we assume that $K$ satisfies (\ref{case1}) for some fixed $\lambda\in(-\infty,0)$. We fix $\varepsilon > 0$, and we assume that $\mu_0 \in \mathcal M^+_{\lambda} \cap \mathcal M^+_{2\lambda-\varepsilon}$. We denote by $(\mu_t)_{t\geq 0}$ the unique $(\mu_0,K,\lambda)$-weak solution to the Smoluchowski equation. We also consider the $(n,K,\mu_0^n)$-Marcus Lushnikov process, for some given initial condition $\mu_0^n=\frac{1}{n} \sum_{i=1}^N \delta_{x_i}$.\medskip

We introduce, for $t \geq 0$, the quantity $E_n(t,x) = G^{\mu^n_t}(x) - G^{\mu_t}(x)$ as defined in (\ref{Intro:MetricEq1}). We take the test function $\phi(v) = \mathds 1_{(0,x]}(v)$. Since $\sup_{v>0}v^{-\lambda} |\phi(v)|= x^{-\lambda}<+\infty$, we deduce that $\phi \in \mathcal H_{\lambda}$. Computing the difference between equations (\ref{Dev:Eq-mu_n-2}) and (\ref{Intro:SmoEq_Weak}), we get
\begin{eqnarray}\label{Dev:Etx}
E_n(t,x) & = & E_n(0,x) + \dfrac{1}{2} \int_0^t \left\langle \mu^n_s(dv)\mu^n_s(dy) - \mu_s(dv)\mu_s(dy), \left(A\mathds 1_{(0,x]}\right)(v,y)K(v,y) \right\rangle ds \nonumber\\
& & -\, \dfrac{1}{2n} \int_0^t \langle \mu^n_s(dv), \left(A\mathds 1_{(0,x]}\right)(v,v) K(v,v) \rangle ds \\
& & +\, \int_0^t \int_{i<j} \int_0^{+\infty} \dfrac{1}{n}\left(A\mathds 1_{(0,x]}\right)\left(X^i_{s-},X^j_{s-}\right) \mathds 1_{\left\lbrace z\leq \frac{K\left(X^i_{s-},X^j_{s-}\right)}{n} \right\rbrace } \mathds 1_{\{j\leq N(s-)\}}\nonumber\\
& & \hspace{8cm} \tilde{J}(ds,d(i,j),dz). \nonumber
\end{eqnarray}
We take the absolute value and integrate against $x^{\lambda-1}dx$ on $(0,+\infty)$:
 \begin{eqnarray}\label{Dev:BoundGral}
d_{\lambda}(\mu^n_t,\mu_t) & \leq & d_{\lambda}(\mu^n_0,\mu_0) + A_1(t) + A_2(t) + A_3(t),
\end{eqnarray}
where
\begin{eqnarray}
A_1(t) & = & \frac{1}{2} \int_0^{+\infty} x^{\lambda-1} \left|   \int_0^t \langle \mu^n_s(dv)\mu^n_s(dy) - \mu_s(dv)\mu_s(dy), \left(A\mathds 1_{(0,x]}\right)(v,y)K(v,y) \rangle  ds\right| dx, \nonumber\\
A_2(t)& = & \dfrac{1}{2n} \int_0^{+\infty} x^{\lambda-1} \left| \int_0^t  \langle \mu^n_s(dv), \left(A\mathds 1_{(0,x]}\right)(v,v) K(v,v) \rangle \, ds\right|dx, \nonumber\\
A_3(t) & = & \int_0^{+\infty} x^{\lambda-1}  \bigg| \frac{1}{n} \int_0^t \int_{i<j} \int_0^{+\infty} \left(A\mathds 1_{(0,x]}\right)\left(X^i_{s-},X^j_{s-}\right)\mathds 1_{\left\lbrace z\leq \frac{K\left(X^i_{s-},X^j_{s-}\right)}{n} \right\rbrace }\nonumber\\
& & \hspace{6.5cm}\mathds 1_{\{j \leq N(s-)\}}\tilde{J}(ds,d(i,j),dz)\bigg| \, dx. \nonumber
\end{eqnarray}
Now we are going to search for a good upper bound for each term.\medskip

\noindent
\underline{\textbf{Term $A_1(t)$.}}\medskip

Similarly to \cite[Lemma 3.5]{Well-Pdnss}. However, in this case we have to argue a little more, since $t\mapsto G^{\mu_t^n}(x)$ is not (even weakly) differentiable due to the jumps of $\mu_t^n$.\medskip

The term $A_1(t)$, according to the symmetry of the kernel, can be written as:
\begin{eqnarray}\label{Dev:Bound1}
A_1(t) = \frac{1}{2}  \int_0^{+\infty}  x^{\lambda-1}\bigg| \int_0^t   \int_0^{+\infty} \int_0^{+\infty} K(v,y) \left[ \mathds 1_{(0,x]}(v+y)- \mathds 1_{(0,x]}(v) -  \mathds 1_{(0,x]}(y) \right] \nonumber \\
\hspace{8cm} \left( \mu^n_s - \mu_s \right)(dv)\left( \mu^n_s + \mu_s \right)(dy) ds\bigg| \,dx. 
\end{eqnarray}
We use the Fubini theorem and Lemma \ref{Dev:lemmaforfubini}:
%===========one only equation split into 3 because of the length===================
$$\int_0^t   \int_0^{+\infty} \int_0^{+\infty} K(v,y) \left[ \mathds 1_{(0,x]}(v+y)- \mathds 1_{(0,x]}(v) -  \mathds 1_{(0,x]}(y) \right] \left( \mu^n_s - \mu_s \right)(dv)\left( \mu^n_s + \mu_s \right)(dy) ds $$
 \begin{eqnarray*}
&&= \, \int_0^t   \int_0^{+\infty} \int_0^{+\infty} \bigg\{K(x-y,y)\mathds 1_{(0,x]}(v+y) -  K(x,y)\mathds 1_{(0,x]}(v)\\ 
&&\hspace{0.8cm} -\,\int_v^{+\infty}\partial_x K(z,y)\left[ \mathds 1_{(0,x]}(z+y)- \mathds 1_{(0,x]}(z) -  \mathds 1_{(0,x]}(y) \right] dz\bigg\}  \left( \mu^n_s - \mu_s \right)(dv)\nonumber\\
&&\hspace{9.6cm}\left( \mu^n_s + \mu_s \right)(dy) ds 
\end{eqnarray*}
\begin{eqnarray*}
&&\hspace{2mm}=\, \int_0^t   \int_0^{+\infty} K(x-y,y) \left[\mathds 1_{x>y} \int_0^{+\infty} \mathds 1_{(0,x-y]}(v) \left( \mu^n_s - \mu_s \right)(dv)\right]\left( \mu^n_s + \mu_s \right)(dy) ds\\
&&\hspace{0.8cm}-\, \int_0^t  \int_0^{+\infty} K(x,y) \left[ \int_0^{+\infty} \mathds 1_{(0,x]}(v) \left( \mu^n_s - \mu_s \right)(dv)\right]\left( \mu^n_s + \mu_s \right)(dy) ds \\
&&\hspace{0.8cm}-\, \int_0^t  \int_0^{+\infty} \int_0^{+\infty} \partial_x K(z,y)\left[ \mathds 1_{(0,x]}(z+y)- \mathds 1_{(0,x]}(z) -  \mathds 1_{(0,x]}(y) \right] \\
&&\hspace{4cm}\left[\int_0^{+\infty} \mathds 1_{(0,z]}(v) \left( \mu^n_s - \mu_s \right)(dv)\right] dz\left( \mu^n_s + \mu_s \right)(dy) ds
\end{eqnarray*}
\begin{eqnarray*}
&&=\, \int_0^t   \int_0^{+\infty} K(x-y,y) \left[\mathds 1_{x>y}\, E_n(s,x-y)\right]\left( \mu^n_s + \mu_s \right)(dy) ds  \\
&&\hspace{0.8cm}-\, \int_0^t  \int_0^{+\infty} K(x,y) \left[ E_n(s,x)\right]\left( \mu^n_s + \mu_s \right)(dy) ds \\
&&\hspace{0.8cm}-\, \int_0^t  \int_0^{+\infty} \int_0^{+\infty} \partial_x K(z,y)\left[ \mathds 1_{(0,x]}(z+y)- \mathds 1_{(0,x]}(z) -  \mathds 1_{(0,x]}(y) \right] \\
&&\hspace{4cm}\left[E_n(s,z)\right] dz\left( \mu^n_s + \mu_s \right)(dy) ds.
\end{eqnarray*}
%===========one only equation split into 3 because of the length===================
According to the bound
\begin{equation}\label{Dev:Bound_Indic} 
\left| \mathds 1_{(0,x]}(z+y)- \mathds 1_{(0,x]}(z) -  \mathds 1_{(0,x]}(y) \right| \leq 2 \,\mathds 1_{(0,x]}(z\wedge y),
\end{equation} 
and using (\ref{case1}), we deduce:
\begin{eqnarray}\label{Dev:UpB2}
A_1(t)&\leq & \frac{\kappa_0}{2}  \int_0^t  \int_0^{+\infty} \int_y^{+\infty}  x^{\lambda-1}  x^{\lambda} |E_n(s,x-y)|\,dx\left( \mu^n_s + \mu_s \right)(dy)\, ds \nonumber\\
&&+\frac{\kappa_0}{2} \int_0^t \int_0^{+\infty}    \int_0^{+\infty}  x^{\lambda-1}  (x+y)^{\lambda} |E_n(s,x)|\,dx  \left( \mu^n_s + \mu_s \right)(dy)\,ds\nonumber\\
&& + \int_0^t  \int_0^{+\infty}   \int_0^{+\infty} |\partial_x K(z,y)||E_n(s,z)| \left[ \int_0^{+\infty} x^{\lambda-1}  \mathds 1_{(0,x]}(z\wedge y)dx  \right]dz \left( \mu^n_s + \mu_s \right)(dy)\, ds.\nonumber
\end{eqnarray}
For the first integral we use the change of variable $x \mapsto w + y$ and $(w+y)^{\lambda-1}\,(w+y)^{\lambda} \leq w^{\lambda-1} y^{\lambda}$. For the second integral $(x+y)^{\lambda} \leq y^{\lambda}$. Finally for the third integral, we observe that $\int_0^{+\infty} x^{\lambda-1} \mathds 1_{(0,x]}(z\wedge y) dx = \frac{(z \wedge y)^{\lambda} }{|\lambda |} \leq \frac{z^{\lambda} + y^{\lambda}}{|\lambda |}$ . Using (\ref{case1}) again, this implies
 \begin{eqnarray}\label{Dev:UpB3}
A_1(t)&\leq & \frac{\kappa_0}{2}  \int_0^t ds\int_0^{+\infty} w^{\lambda-1} |E_n(s,w)|\, dw \int_0^{+\infty}   y^{\lambda}  \left( \mu^n_s + \mu_s \right)(dy)\nonumber\\
&& + \frac{\kappa_0}{2}  \int_0^t ds\int_0^{+\infty} x^{\lambda-1} |E_n(s,x)|\, dx \int_0^{+\infty}   y^{\lambda}  \left( \mu^n_s + \mu_s \right)(dy)\nonumber\\
&& + \frac{\kappa_1}{|\lambda|} \int_0^t ds\int_0^{+\infty} z^{\lambda-1} |E_n(s,z)|\, dz \int_0^{+\infty}   y^{\lambda}  \left( \mu^n_s + \mu_s \right)(dy).\nonumber
\end{eqnarray}
The resulting bound for $A_1(t)$ is:
 \begin{equation}\label{Dev:Bound_A1}
A_1(t) \leq  \left(\kappa_0 + \frac{\kappa_1}{|\lambda|}\right)\int_0^t d_{\lambda}(\mu^n_s,\mu_s)\, M_{\lambda}(\mu^n_s + \mu_s)\, ds.
\end{equation}

\noindent
\underline{\textbf{Term $A_2(t)$.}}\medskip

We use $\left|\left(A\mathds 1_{(0,x]}\right)(v,v)\right| = \left|\mathds 1_{(0,x]}(2v) - 2\, \mathds 1_{(0,x]}(v)\right| = \mathds 1_{\{0 < v \leq \frac{x}{2} \}}  + 2\,\mathds 1_{\{\frac{x}{2} < v \leq x \}} \leq 2\,\mathds 1_{\{v \leq x \}}$. This gives
 \begin{eqnarray}\label{Dev:Bound_A2}
A_2(t) & \leq & \frac{1}{n} \int_0^{+\infty} x^{\lambda-1} \int_0^t \int_0^{+\infty} K(v,v) \mathds 1_{\{v \leq x \}} \mu^n_s(dv)\,ds\, dx\nonumber\\
& \leq & \frac{1}{n} \int_0^{+\infty} \int_0^t \kappa_0(2v)^{\lambda} \frac{v^{\lambda}}{|\lambda |} \mu^n_s(dv)\,ds \nonumber\\
& \leq  & \frac{2^{\lambda}\kappa_0}{n\,|\lambda|} \int_0^t M_{2\lambda}(\mu^n_s)\, ds.
\end{eqnarray}
We used (\ref{case1}).\medskip

\noindent 
\underline{\textbf{Term $A_3(t)$.}}\medskip

We will bound the expectation of this term using its bracket, for this we consider:
 \begin{eqnarray*}
 && \mathbb E \left[  \left( \frac{1}{n} \, \int_0^t \int_{i<j} \int_0^{+\infty} (A\mathds 1_{(0,x]})\left(X^i_{s-},X^j_{s-}\right) \mathds 1_{\left\lbrace z\leq \frac{K\left(X^i_{s-},X^j_{s-}\right)}{n} \right\rbrace }\,\mathds 1_{\{j \leq N(s-)\}}\tilde{J}(ds,d(i,j),dz)\right)^2\right] \\
&& \hspace{1.3cm} = \mathbb E \left[ \int_0^t \frac{1}{n^2} \sum_{i<j\leq N(s)} \frac{K\left(X^i_{s},X^j_{s}\right)}{n}\left[\mathds 1_{(0,x]} \left(X^i_{s}+X^j_{s}\right) - \mathds 1_{(0,x]}\left(X^i_{s}\right) -  \mathds 1_{(0,x]}\left(X^j_{s}\right)\right]^2 ds \right]  \\
&&\hspace{1.3cm} \leq \frac{4}{n}   \mathbb E \left[ \int_0^t \sum_{i<j\leq N(s)} \frac{K\left(X^i_{s},X^j_{s}\right)}{n^2} \mathds 1_{(0,x]}\left(X^i_{s}\wedge X^j_{s}\right)ds \right]  \\ 
&&\hspace{1.3cm} \leq \frac{2}{n} \mathbb E \left[ \int_0^t  \left\langle \mu^n_s(dv) \mu^n_s(dy)\,,\,K(v,y)\,\left[ \mathds 1_{(0,x]}(v) + \mathds 1_{(0,x]}(y)\right]  \right\rangle\, ds \right] \\
&&\hspace{1.3cm} \leq \frac{4\, \kappa_0}{n}  \mathbb E \left[ \int_0^t \left\langle \mu^n_s(dv) \mu^n_s(dy)\,,\,(v+y)^{\lambda}\,\mathds 1_{(0,x]}(v)   \right\rangle \,ds \right]. 
\end{eqnarray*}
We have used (\ref{Dev:Bound_Indic}), a symmetry argument then the bound $\mathds 1_{(0,x]} (v\vee y) \leq \mathds 1_{(0,x]} (v) + \mathds 1_{(0,x]} (y)$ and finally (\ref{case1}). We consider now the submartingale (absolute value of a martingale):
\[ S_t(x) = \left| \frac{1}{n} \int_0^t \int_{i<j} \int_0^{+\infty} \left(A\mathds 1_{(0,x]}\right)\left(X^i_{s-},X^j_{s-}\right)\mathds 1_{\left\lbrace z\leq \frac{K\left(X^i_{s-},X^j_{s-}\right)}{n} \right\rbrace }\mathds 1_{\{j \leq N(s-)\}}\,\tilde{J}(ds,d(i,j),dz)\right|.\]
According to the Cauchy-Schwartz and Doob inequalities we have:
\[ \mathbb E\left[\sup_{r\in[0,t]}S_r(x)\right] \leq \left( \mathbb E\left[\sup_{r\in[0,t]} \left(S_r(x)\right)^2\right] \right)^{\frac{1}{2}} \leq 2 \left(\mathbb E\left[\left(S_t(x)\right)^2\right]\right)^{\frac{1}{2}}. \]
Therefore, we obtain the following bound for the expectation of $A_3(t)$:
\begin{eqnarray}\label{Dev:Bound_A3_1}
\mathbb E\left[\sup_{s\in[0,t]}A_3(s)\right] &\leq &\frac{4\sqrt{\kappa_0} }{\sqrt{n}} \int_0^{+\infty} x^{\lambda-1}\nonumber\\ 
& & \left\lbrace  \mathbb E \left[ \int_0^t \left\langle \mu^n_s(dv)\mu^n_s(dy)\,,\,(v+y)^{\lambda}\,\mathds 1_{(0,x]}(v)   \right\rangle  \,ds \right] \right\rbrace ^{\frac{1}{2}} dx.
\end{eqnarray}
Following the value of $x$ we use different bounds: \medskip

On the one hand, for $x\leq 1$ we have $\mathds 1_{(0,x]}(v) \leq \left( \frac{v}{x}\right) ^{2\lambda-\varepsilon}$ and using the bound $(v + y)^{\lambda}v^{2\lambda-\varepsilon} \leq v^{2\lambda-\varepsilon} y^{\lambda}$, we obtain:
\begin{eqnarray}\label{Dev:Bound_A3:x_0to1}
&&\int_0^{1} x^{\lambda-1} \left\lbrace  \mathbb E \left[ \int_0^t \big\langle \mu^n_s(dv)\mu^n_s(dy)\,,\,(v+y)^{\lambda}\,\mathds 1_{(0,x]}(v) \big\rangle \,ds \right] \right\rbrace ^{\frac{1}{2}} dx\hspace{3cm} \nonumber\\
&& \hspace{4cm} \leq \int_0^{1} x^{\lambda-1} \left\lbrace  \mathbb E \left[ \int_0^t \left\langle \mu^n_s(dv)\mu^n_s(dy)\,,\,\frac{v^{2\lambda-\varepsilon} y^{\lambda}}{x^{2\lambda-\varepsilon} } \right\rangle \,ds \right] \right\rbrace^{\frac{1}{2}}dx  \nonumber\\
&& \hspace{4cm}= \int_0^{1} x^{\frac{\varepsilon}{2}-1} dx \,\, \left\lbrace  \mathbb E \left[ \int_0^t \big\langle \mu^n_s(dv)\mu^n_s(dy)\,,\,v^{2\lambda-\varepsilon} y^{\lambda}   \big\rangle \,ds \right] \right\rbrace ^{\frac{1}{2}} \nonumber\\
&&\hspace{4cm} = \frac{2}{\varepsilon} \left\lbrace  \mathbb E \left[ \int_0^t M_{\lambda}(\mu^n_s)\, M_{2\lambda-\varepsilon} (\mu^n_s) \,ds \right] \right\rbrace ^{\frac{1}{2}}.
\end{eqnarray}

On the other hand, for $x>1$ we have $\mathds 1_{(0,x]}(v) \leq \left( \frac{v}{x}\right) ^{\lambda} $ and using the bound $(v + y)^{\lambda}v^{\lambda} \leq v^{\lambda} y^{\lambda}$, we obtain:
\begin{eqnarray}\label{Dev:Bound_A3:x_1toinf}
&& \int_1^{+\infty} x^{\lambda-1} \left\lbrace  \mathbb E \left[ \int_0^t \big\langle \mu^n_s(dv)\mu^n_s(dy)\,,\,(v+y)^{\lambda}\,\mathds 1_{(0,x]}(v)   \big\rangle \,ds \right] \right\rbrace ^{\frac{1}{2}} dx \hspace{3cm} \nonumber\\
&&\hspace{4cm} \leq\int_1^{+\infty} x^{\lambda-1} \left\lbrace  \mathbb E \left[ \int_0^t \left\langle \mu^n_s(dv)\mu^n_s(dy)\,,\,\frac{v^{\lambda}y^{\lambda}}{x^{\lambda}} \right\rangle \,ds \right] \right\rbrace ^{\frac{1}{2}} dx \nonumber\\
&&\hspace{4cm} = \int_1^{+\infty} x^{\frac{\lambda}{2}-1} dx \,\, \left\lbrace  \mathbb E \left[ \int_0^t \big\langle \mu^n_s(dv)\mu^n_s(dy)\,,\,v^{\lambda}y^{\lambda}   \big\rangle \,ds \right] \right\rbrace ^{\frac{1}{2}} \nonumber\\
&&\hspace{4cm} = \frac{2}{|\lambda|} \left\lbrace  \mathbb E \left[ \int_0^t \left[M_{\lambda}(\mu^n_s)\right]^2 \,ds \right] \right\rbrace ^{\frac{1}{2}}.
\end{eqnarray}

Then, writing the right-hand side integral of (\ref{Dev:Bound_A3_1}) as the sum of the integrals on $x\in(0,1]$ and $x\in(1,+\infty)$, gathering (\ref{Dev:Bound_A3:x_0to1}) and (\ref{Dev:Bound_A3:x_1toinf}), we get
\begin{eqnarray}\label{Dev:Bound_A3}
\mathbb E\left[\sup_{s\in[0,t]}A_3(s)\right] & \leq & \frac{8\sqrt{\kappa_0} }{\sqrt{n}} \Bigg\{ \frac{1}{\varepsilon} \left(  \mathbb E \left[ \int_0^t M_{\lambda}(\mu^n_s) M_{2\lambda-\varepsilon} (\mu^n_s) \,ds \right]\right) ^{\frac{1}{2}} \nonumber \\
& &\hspace{5cm} + \frac{1}{|\lambda|} \left(  \mathbb E \left[ \int_0^t \left[M_{\lambda}(\mu^n_s)\right]^2 \,ds\right] \right)^{\frac{1}{2}}\Bigg\}.
\end{eqnarray}

\noindent
\underline{\textbf{Conclusion.}}\medskip

Gathering (\ref{Dev:BoundGral}), (\ref{Dev:Bound_A1}), (\ref{Dev:Bound_A2}) and (\ref{Dev:Bound_A3}), we have:
\begin{eqnarray}\label{Dev:Final_casneg}
\mathbb E\left[\sup_{s\in[0,t]} d_{\lambda}(\mu^n_s,\mu_s)  \right]  & \leq & \mathbb E\left[ d_{\lambda}(\mu^n_0,\mu_0)  +\sup_{s\in[0,t]}A_1(s) + \sup_{s\in[0,t]}A_2(s)+\sup_{s\in[0,t]}A_3(s)\right] \nonumber\\
& \leq & d_{\lambda}(\mu^n_0,\mu_0)  + \left(\kappa_0 + \frac{\kappa_1}{|\lambda|}\right)\int_0^t \mathbb E\left[ d_{\lambda}(\mu^n_s,\mu_s)\, M_{\lambda}(\mu^n_s + \mu_s)\right] ds \nonumber\\
& & +\, \frac{2^{\lambda}\kappa_0}{n\,|\lambda|} \int_0^t \mathbb E\left[ M_{2\lambda}(\mu^n_s)\right] \, ds\nonumber\\
& & +\, \frac{8\sqrt{\kappa_0} }{\sqrt{n}} \Bigg\{ \frac{1}{\varepsilon} \left(  \mathbb E \left[ \int_0^t M_{\lambda}(\mu^n_s)\, M_{2\lambda-\varepsilon} (\mu^n_s) \,ds \right]\right) ^{\frac{1}{2}}\nonumber\\
& & \hspace{1.7cm} + \frac{1}{|\lambda|} \left(  \mathbb E \left[ \int_0^t \left[M_{\lambda}(\mu^n_s)\right]^2 \,ds\right] \right) ^{\frac{1}{2}}\Bigg\}.\nonumber
\end{eqnarray}
According to Proposition \ref{Prop:bounded_moments} --\textit{(\ref{Prop:bounded_moments_neg})}, $M_{\alpha}(\mu^n_t+\mu_t) \leq M_{\alpha}(\mu^n_0+\mu_0)$ a.s. for any $\alpha \in (-\infty,0)$. Since $\mu^n_0$ is deterministic, we get:
\begin{eqnarray}\label{Dev:pre-Gronwall}
\mathbb E\left[\sup_{s\in[0,t]} d_{\lambda}(\mu^n_t,\mu_t)  \right]  & \leq & d_{\lambda}(\mu^n_0,\mu_0)  + \left(\kappa_0 + \frac{\kappa_1}{|\lambda|}\right)M_{\lambda}(\mu^n_0+\mu_0)\int_0^t \mathbb E\left[ d_{\lambda}(\mu^n_s,\mu_s)\, \right] ds\nonumber\\
& & + \frac{2^{\lambda}\kappa_0}{n\,|\lambda|} M_{2\lambda}(\mu^n_0)\,t + \frac{8\sqrt{\kappa_0} }{\sqrt{n}} \left[ \frac{1}{\varepsilon} \big( M_{\lambda}(\mu^n_0)\, M_{2\lambda-\varepsilon} (\mu^n_0)\big)^{\frac{1}{2}} + \frac{1}{|\lambda|} M_{\lambda}(\mu^n_0)\right] t^{\frac{1}{2}}.
\end{eqnarray}
Finally, since $\sqrt{ab} \leq a + b$ and since $M_{2\lambda}(\mu_0^n) \leq  M_{\lambda}(\mu_0^n) + M_{2\lambda-\varepsilon}(\mu_0^n)$, we use the Gronwall lemma to obtain
\begin{eqnarray}
\mathbb E\left[\sup_{t\in[0,T]} d_{\lambda}(\mu^n_t,\mu_t)  \right] & \leq & \left[ d_{\lambda}(\mu^n_0,\mu_0)  + \frac{C_1}{\sqrt{n}}\,M_{\lambda }(\mu^n_0) + \frac{C_2}{\sqrt{n}}\, M_{2\lambda-\varepsilon} (\mu^n_0)\big) \right]\nonumber\\ 
& & \hspace{3cm} \times  \exp\bigg[T\left(\kappa_0 + \frac{\kappa_1}{|\lambda|}\right)M_{\lambda}(\mu^n_0+\mu_0)\bigg],
\end{eqnarray}
where $C_1 = \frac{2^{\lambda}T\kappa_0}{|\lambda|} +   \frac{8(\varepsilon+ |\lambda|)}{\varepsilon\,|\lambda|}\sqrt{T\kappa_0}$ and $C_2=\frac{2^{\lambda}T\kappa_0}{|\lambda|} + \frac{8}{\varepsilon}\sqrt{T\kappa_0}$. \medskip

This concludes the proof of Theorem \ref{theorem} under (\ref{case1}).

\section{Positive Case}\label{Positive_Case}
\setcounter{equation}{0}
In the whole section, we assume that $K$ satisfies (\ref{case2}) for some fixed $\lambda\in(0,1]$. We fix $\varepsilon > 0$, and we assume that $\mu_0 \in \mathcal M^+_{0} \cap \mathcal M^+_{\gamma+\varepsilon}$ where $\gamma = \max\{2\lambda,\,4\lambda-1\}$. We denote by $(\mu_t)_{t\geq 0}$ the unique $(\mu_0,K,\lambda)$-weak solution to the Smoluchowski equation. We also consider the $(n,K,\mu_0^n)$-Marcus Lushnikov process, for some given initial condition $\mu_0^n=\frac{1}{n} \sum_{i=1}^N \delta_{x_i}$.\medskip

We assume without loss of generality, for $\lambda \in (0,1/2)$, that $\varepsilon < \frac{1}{2} - \lambda$. Indeed, if $\varepsilon \geq \frac{1}{2} - \lambda$, it suffices to consider $\tilde{\varepsilon} < \frac{1}{2} - \lambda$, to apply Theorem \ref{theorem} with $\tilde{\varepsilon}$, and to use the bound $M_{2\lambda+\tilde{\varepsilon}}(\mu_0^n + \mu_0) \leq M_{0}(\mu_0^n + \mu_0) + M_{2\lambda + \varepsilon}(\mu_0^n + \mu_0)$ to conclude. \medskip

We first present a lemma of which the proof is developed in the appendix.
\begin{lemma} \label{lemma_integrals}
We introduce, for $x\in(0,+\infty)$, the following function: 
\begin{equation}\label{theta}
\theta_{(x)}^n = \dfrac{1}{\sqrt{n}}\mathds 1_{(0,1]}(x) + \dfrac{x^{-2\lambda - \varepsilon}}{\sqrt{n}}\,\mathds 1_{(1,+\infty)}(x).
\end{equation}
Then,
\begin{enumerate}[(i)]
\item \label{lemma:1} $\int_0^{+\infty} x^{\lambda - 1}\theta_{(x)}^n dx \leq \frac{2}{\lambda \sqrt{n}}$,
\item \label{lemma:2} $\int_0^{+\infty} x^{2\lambda - 1}\theta_{(x)}^n dx \leq \frac{\lambda + \varepsilon}{\lambda\varepsilon\, \sqrt{n}}$,
\item \label{lemma:3} for $(v,y)\in (0,+\infty)^2$
\begin{eqnarray*}
&&  v^{\lambda}\int_0^{+\infty}  \frac{x^{\lambda-1}}{\theta_{(x)}^n}\left(\mathds 1_{x <v\wedge y} + \mathds 1_{ v\vee y < x <v + y} \right)  dx \leq \frac{2\sqrt{n}}{\lambda}  v^{\lambda}y^{\lambda} \\
&& \hspace{1cm} +\sqrt{n}\left(2^{2\lambda+\varepsilon} + \frac{1}{\lambda}\right)\left[(v\wedge y)^{2\lambda}(v\vee y)^{2\lambda+\varepsilon}\mathds 1_{\lambda\in(0,1/2)} + (v\wedge y)(v\vee y)^{4\lambda+\varepsilon-1}\mathds 1_{\lambda\in[1/2,1]} \right].
\end{eqnarray*}
\end{enumerate}
\end{lemma}

We set $E_n(t,x) = F^{\mu^ n_t}(x) - F^{\mu_t}(x)$ as defined in (\ref{Intro:MetricEq1}), for $x\in (0,+\infty)$. We take the test function $\phi(v) = \mathds 1_{(x,+\infty)}(v)$. Since $\sup_{v>0}\frac{|\phi(v)|}{(1+v)^{\lambda} }= (1 + x)^{-\lambda}<+\infty$, we deduce that $\phi \in \mathcal H_{\lambda}$. Again, computing the difference between equations (\ref{Dev:Eq-mu_n-2}) and (\ref{Intro:SmoEq_Weak}) and using a symmetry argument for the first integral, we get
\begin{eqnarray}\label{Dev:EtxCasPos}
E_n(t,x) & = & E_n(0,x) + \dfrac{1}{2} \int_0^t \langle \left( \mu^n_s - \mu_s \right)(dv)\left( \mu^n_s + \mu_s \right)(dy), \left(A\mathds 1_{(x,+\infty)}\right)(v,y)K(v,y) \rangle ds \nonumber\\
& & -\, \dfrac{1}{2n} \int_0^t \langle \mu^n_s(dv),  \left(A\mathds 1_{(x,+\infty)}\right)(v,v) K(v,v) \rangle ds \\
& & +\, \int_0^t \int_{i<j} \int_0^{+\infty} \dfrac{1}{n}  \left(A\mathds 1_{(x,+\infty)}\right)\left(X^i_{s-},X^j_{s-}\right) \mathds 1_{\left\lbrace z\leq \frac{K\left(X^i_{s-},X^j_{s-}\right)}{n} \right\rbrace } \mathds 1_{\{ j \leq N(s-)\}} \nonumber\\
& & \hspace{8cm} \,\tilde{J}(ds,d(i,j),dz). \nonumber
\end{eqnarray}
According to Lemma \ref{Dev:lemmaforfubini}, we can write the first integral as:
$$\int_0^t   \int_0^{+\infty} \int_0^{+\infty} K(v,y)  \left(A\mathds 1_{(x,+\infty)}\right)(v,y) \left( \mu^n_s - \mu_s \right)(dv)\left( \mu^n_s + \mu_s \right)(dy) ds \hspace{3.5cm}$$
\begin{eqnarray}
&=&\int_0^t   \int_0^{+\infty} \int_0^{+\infty} \bigg\{ \mathds 1_{x>y}K(x-y,y)\mathds 1_{(x,+\infty)}(v+y) -  K(x,y)\mathds 1_{(x,+\infty)}(v) \nonumber\\ 
&&\hspace{1cm}+\,\int_0^{v}\partial_x K(z,y) \left(A\mathds 1_{(x,+\infty)}\right)(z,y)\,dz\bigg\} \left( \mu^n_s - \mu_s \right)(dv)\left( \mu^n_s + \mu_s \right)(dy) ds \nonumber
\end{eqnarray}
\begin{eqnarray}
&=& \int_0^t \int_0^{+\infty}  K(x-y,y) \left[\mathds 1_{x>y} \int_0^{+\infty} \mathds 1_{(x-y,+\infty)}(v) \left( \mu^n_s - \mu_s \right)(dv)\right]\left( \mu^n_s + \mu_s \right)(dy) ds\nonumber\\
&&-\,\int_0^t  \int_0^{+\infty} K(x,y) \left[ \int_0^{+\infty} \mathds 1_{(x,+\infty)}(v) \left( \mu^n_s - \mu_s \right)(dv)\right]\left( \mu^n_s + \mu_s \right)(dy) ds \nonumber\\
&& + \int_0^t \int_0^{+\infty} \int_0^{+\infty} \partial_x K(z,y) \left(A\mathds 1_{(x,+\infty)}\right)(z,y) \nonumber\\
&& \hspace{3cm}\left[\int_0^{+\infty} \mathds 1_{(z,+\infty)}(v) \left( \mu^n_s - \mu_s \right)(dv)\right] dz\left( \mu^n_s + \mu_s \right)(dy) ds.\nonumber
\end{eqnarray}
Recalling that $ E_n(s,x) =  \int_0^{+\infty} \mathds 1_{(x,+\infty)}(v) \left( \mu^n_s - \mu_s \right)(dv)$, we deduce that,
\begin{eqnarray}\label{Dev:EtxCasPos_2}
E_n(t,x) & = & E_n(0,x) + \frac{1}{2}\int_0^t \left[ \overline{B}_1(s,x) + \overline{B}_2(s,x) + \overline{B}_3(s,x)\right] ds \nonumber\\
& & +\, \int_0^t \int_{i<j} \int_0^{+\infty} \dfrac{1}{n} \left(A\mathds 1_{(x,+\infty)}\right)\left(X^i_{s-},X^j_{s-}\right) \mathds 1_{\left\lbrace z\leq \frac{K\left(X^i_{s-},X^j_{s-}\right)}{n} \right\rbrace } \mathds 1_{\{ j \leq N(s-)\}} \\
& & \hspace{8cm}\,\tilde{J}(ds,d(i,j),dz), \nonumber
\end{eqnarray}
where:
\begin{eqnarray}
\overline{B}_1(s,x) & = &   \int_0^{+\infty} \big[ \mathds 1_{x>y} K(x-y,y) E_n(s,x-y)  - E_n(s,x) K(x,y) \big]  \left( \mu^n_s + \mu_s \right)(dy),\nonumber\\
\overline{B}_2(s,x) & = & \int_0^{+\infty} \int_0^{+\infty} \partial_x K(z,y) \left(A\mathds 1_{(x,+\infty)}\right)(z,y) E_n(s,z) dz \left( \mu^n_s + \mu_s \right)(dy),\nonumber\\
\overline{B}_3(s,x)& = & -\dfrac{1}{n} \int_0^{+\infty} K(v,v) \left[\mathds 1_{(x,+\infty)}(2v) - 2\, \mathds 1_{(x,+\infty)}(v) \right] \mu^n_s(dv).\nonumber
\end{eqnarray}

Now, we apply the It\^o formula to $\varphi_{\theta}(E_n(t,x))$, where $\varphi_{\theta}(\cdot)\in\mathcal C^2(\mathbb R)$ is an approximation of the absolute value function $|\cdot|$. This function is chosen in such a way that:
\begin{equation}\label{Dev:CondPhi}
\left\{
\begin{array}{lll}
\varphi_{\theta}(u) = |u|\,\, \textrm{ if }|u| > \theta;  &|u| \leq \varphi_{\theta}(u) \leq |u| + \theta \,\,\, \forall u \in \mathbb R; \\[2mm]
|\varphi'_{\theta}(u)| \leq 1 \,\,\,\, \forall u \in \mathbb R;  & sgn(u\varphi'_{\theta}(u))=1  \,\, \forall u \in \mathbb R_* ;\\[2mm]
|\varphi''_{\theta}(u)| \leq \dfrac{2}{\theta} \mathds 1_{\{|u| < \theta\}} \,\,\, \forall u \in \mathbb R.  &
\end{array}\right.
\end{equation}
Furthermore, we consider for $\theta$ the function defined by (\ref{theta}). We fix $x\in(0,+\infty)$ and apply the It\^o formula to $\varphi_{\theta_{(x)}^n}(E_n(t,x))$ (see for exemple \cite{Jacod}),
\begin{eqnarray}\label{Dev:ItoCasPos}
\varphi_{\theta_{(x)}^n}(E_n(t,x))& = & \varphi_{\theta_{(x)}^n}(E_n(0,x))\nonumber \\
& & + \frac{1}{2} \int_0^t  \left[ \overline{B}_1(s,x) +\overline{B}_2(s,x) + \overline{B}_3(s,x)\right]  \varphi_{\theta_{(x)}^n}'(E_n(s,x))ds \\
& & + M(t,x) + \overline{B}_4(t,x), \nonumber
\end{eqnarray}
where
\begin{eqnarray} 
M(t,x) & = & \int_0^t \int_{i<j} \int_0^{+\infty} \dfrac{1}{n} \left(A\mathds 1_{(x,+\infty)}\right)\left(X^i_{s-},X^j_{s-}\right) \mathds 1_{\left\lbrace z\leq \frac{K\left(X^i_{s-},X^j_{s-}\right)}{n} \right\rbrace } \mathds 1_{\{j\leq N(s-)\}}\nonumber\\
& & \hspace{5.5cm} \varphi_{\theta_{(x)}^n}'(E_n(s-,x))\tilde{J}(ds,d(i,j),dz),\nonumber\\
\overline{B}_4(t,x) & = & \int_0^t \int_{i<j} \int_0^{+\infty}  \bigg\{ \varphi_{\theta_{(x)}^n}\left(E_n(s-,x) + \dfrac{1}{n}\left(A\mathds 1_{(x,+\infty)}\right)\left(X^i_{s-},X^j_{s-}\right)\right) - \varphi_{\theta_{(x)}^n}(E_n(s-,x))  \nonumber \\
&& \hspace{1cm} -\dfrac{1}{n}\left(A\mathds 1_{(x,+\infty)}\right) \left(X^i_{s-},X^j_{s-}\right) \varphi_{\theta_{(x)}^n}'(E_n(s-,x))\bigg\}\mathds 1_{\left\lbrace z\leq \frac{K\left(X^i_{s-},X^j_{s-}\right)}{n} \right\rbrace }\mathds 1_{\{j\leq N(s-)\}}  \nonumber \\
& & \hspace{10cm}J(ds,d(i,j),dz). \nonumber
\end{eqnarray}
Observe that, for all $x\geq 0$, $M(t,x)$ is a martingale whose expectation is equal to zero.\medskip

Now, we study the $\theta^n$-approximation of $d_{\lambda}(\mu^n_t,\mu_t)$: $\int_0^{+\infty} x^{\lambda-1} \varphi_{\theta_{(x)}^n}(E_n(t,x)) dx$. According to (\ref{Dev:CondPhi}) and Lemma \ref{lemma_integrals} --\textit{(\ref{lemma:1})}, we have:
\begin{equation}\label{Dev:Convergence}
d_{\lambda}(\mu_s^n,\mu_s) \leq \int_0^{+\infty} x^{\lambda-1} \varphi_{\theta_{(x)}^n}(E_n(s,x)) dx \leq d_{\lambda}(\mu_s^n,\mu_s) + \frac{2}{\lambda \sqrt{n}}.
\end{equation}
Consider (\ref{Dev:ItoCasPos}), integrate each term against $x^{\lambda-1}dx$ on $(0+\infty)$, take the expectation:
\begin{eqnarray}\label{Dev:afterIto}
\mathbb E \left[d_{\lambda}(\mu_t^n,\mu_t) \right]& \leq & \int_0^{+\infty} x^{\lambda-1} \mathbb E \left[\varphi_{\theta_{(x)}^n}(E_n(t,x))\right] dx \nonumber \\
& = &  \int_0^{+\infty} x^{\lambda-1}\varphi_{\theta_{(x)}^n}(E_n(0,x)) dx + \mathbb E \left[ B_1(t) +  B_2(t) +  B_3(t) + B_4(t)\right],
\end{eqnarray}
where
\begin{eqnarray*}
B_1(t) & = &  \frac{1}{2} \int_0^{+\infty} \int_0^t  x^{\lambda-1} \overline{B}_1(s,x) \varphi_{\theta_{(x)}^n}'(E_n(s,x)) \,ds\,dx, \\
B_2(t) & = &  \frac{1}{2} \int_0^{+\infty} \int_0^t  x^{\lambda-1} \overline{B}_2(s,x) \varphi_{\theta_{(x)}^n}'(E_n(s,x)) \,ds\,dx, \\
B_3(t) & = &  \frac{1}{2}\int_0^{+\infty} \int_0^t  x^{\lambda-1} \overline{B}_3(s,x) \varphi_{\theta_{(x)}^n}'(E_n(s,x)) \,ds\,dx,\\
B_4(t) & = &  \int_0^{+\infty}  x^{\lambda-1} \overline{B}_4(t,x) \,dx.
\end{eqnarray*}
We now study each term separately.\medskip

\noindent
\textbf{\underline{Term $B_1(t)$.}}\medskip

We use the Fubini theorem to obtain: 
\begin{eqnarray}
B_1(t) & = & \frac{1}{2} \int_0^t \int_0^{+\infty}\bigg[\int_0^{+\infty}\mathds 1_{x>y}\, x^{\lambda-1}\, \varphi'_{\theta_{(x)}^n}\left(E_n(s,x)\right) E_n(s,x-y) K(x-y,y) dx \nonumber\\
& &\hspace{1cm}-\int_0^{+\infty} x^{\lambda-1}\, \varphi'_{\theta_{(x)}^n}\left(E_n(s,x)\right)  E_n(s,x) K(x,y) \,dx \bigg]\left( \mu^n_s + \mu_s \right)(dy)\,ds. \nonumber
\end{eqnarray}
Recalling (\ref{Dev:CondPhi}), we immediately deduce that $\varphi'_{\theta_{(x)}^n}\left(E_n(s,x)\right)\,E_n(s,x-y) \leq \left| E_n(s,x-y)  \right|$, and $\varphi'_{\theta_{(x)}^n}\left(E_n(s,x)\right)\,E_n(s,x) = \left| \varphi'_{\theta_{(x)}^n} \left(E_n(s,x)\right)\right| \,\left| E_n(s,x)\right|$. Therefore, using the change of variable $x \mapsto u+y$ in the first integral, we get:
\begin{eqnarray}
B_1(t) & \leq & \frac{1}{2} \int_0^t \int_0^{+\infty}\bigg[\int_0^{+\infty} (u+y)^{\lambda-1}\, \left| E_n(s,u)\right| K(u,y)\, du \nonumber\\
&  & \hspace{0.6cm}-\int_0^{+\infty} x^{\lambda-1}\, \left|\varphi'_{\theta_{(x)}^n}\left(E_n(s,x)\right)\right|\,  \left| E_n(s,x)\right| K(x,y) \,dx\bigg] \left( \mu^n_s + \mu_s \right)(dy)\,ds\nonumber\\
& = & \frac{1}{2} \int_0^t \int_0^{+\infty}\int_0^{+\infty}  K(z,y)\,\left| E_n(s,z)\right| \left[(z+y)^{\lambda-1}\,- \left|\varphi'_{\theta_{(z)}^n}\left(E_n(s,z)\right)\right| \,z^{\lambda-1}\right] dz \nonumber\\
&  & \hspace{8.8cm} \left( \mu^n_s + \mu_s \right)(dy)\,ds.\nonumber
\end{eqnarray}
Recall again (\ref{Dev:CondPhi}). Since $|E_n(s,z)|\geq \theta_{(z)}^n$ implies $\left|\varphi'_{\theta_{(z)}^n} \left(E_n(s,z)\right)\right|=1$, and since $(z+y)^{\lambda-1}-z^{\lambda-1} \leq 0$, 
\begin{eqnarray*}
\left| E_n(s,z)\right| \left[(z+y)^{\lambda-1} - |\varphi'_{\theta_{(z)}^n}\left(E_n(s,z)\right)| \, z^{\lambda-1}\right]&  \leq & \left| E_n(s,z)\right|  (z+y)^{\lambda-1} \mathds 1_{\left\lbrace |E_n(s,z)|< \theta_{(z)}^n \right\rbrace}\\
& \leq & \theta_{(z)}^n (z+y)^{\lambda-1}.
\end{eqnarray*}
Therefore, using (\ref{case2}):
\begin{eqnarray*}
B_1(t)& \leq & \frac{\kappa_0}{2} \int_0^t \int_0^{+\infty}\int_0^{+\infty}\theta_{(z)}^n \,(z+y)^{2\lambda-1} \,dz \left( \mu^n_s + \mu_s \right)(dy)\,ds  \\
& \leq & \frac{\kappa_0}{2}\int_0^t \int_0^{+\infty} \left[ \int_0^{+\infty} \theta_{(z)}^n z^{2\lambda-1} dz + y^{\lambda} \int_0^{+\infty}\theta_{(z)}^n z^{\lambda-1} dz \right] \left( \mu^n_s + \mu_s \right)(dy)\,ds. 
\end{eqnarray*}
We used $(z+y)^{2\lambda-1} =(z+y)^{\lambda} (z+y)^{\lambda-1} \leq (z^{\lambda}+y^{\lambda})\, z^{\lambda-1}$. Finally, according to Lemma  \ref{lemma_integrals}--\textit{(\ref{lemma:1})} and \textit{(\ref{lemma:2})}, we get:
\begin{eqnarray}\label{Dev:Bound_B1}
B_1(t)& \leq & \frac{\kappa_0}{2}\int_0^t \int_0^{+\infty} \left[ \frac{2(\lambda+\varepsilon)}{\lambda\varepsilon\sqrt{n}} \left(1 + y^{\lambda}\right) \right] \left( \mu^n_s + \mu_s \right)(dy)\,ds\nonumber\\
& \leq &\frac{\kappa_0(\lambda+\varepsilon)}{\lambda\varepsilon\sqrt{n}} \int_0^t \left[ M_{0}(\mu_s^n + \mu_s) + M_{\lambda}(\mu_s^n + \mu_s)\right]ds. 
\end{eqnarray}

\noindent
\textbf{\underline{Term $B_2(t)$.}}\medskip

First, observe that
\begin{eqnarray}\label{Dev:Bound_IndicCasPos} 
\left|\left(A\mathds 1_{(x,+\infty)}\right)(z,y)\right| & = & \left|\mathds 1_{(x,+\infty)}(z+y)- \mathds 1_{(x,+\infty)}(z) -  \mathds 1_{(x,+\infty)}(y)\right| \nonumber \\
& = & \mathds 1_{\{x \in(0,z\wedge y)\}} + \,\mathds 1_{\{x \in(z \vee y , z+y)\}}\,,
\end{eqnarray}
whence,
\begin{eqnarray}\label{Bound_int_inf}
\int_0^{+\infty} x^{\lambda-1} \left|\left(A\mathds 1_{(x,+\infty)}\right)(z,y)\right| dx  & = & \int_{0}^{z\wedge y} x^{\lambda-1} dx + \int_{z \vee y}^{z + y} x^{\lambda-1} dx \nonumber\\
&\leq &\frac{2}{\lambda}(z \wedge y)^{\lambda}.
\end{eqnarray}
Thus, recalling (\ref{Dev:CondPhi}), we get:
\begin{eqnarray}\label{Dev:Bound_B2}
B_2(t)& \leq &   \frac{1}{2}   \int_0^t\int_0^{+\infty} \int_0^{+\infty} |E_n(s,z)|\, |\partial_x K(z,y)|\, \left(\frac{2}{\lambda}\left( z \wedge y\right)^{\lambda}\right) \left( \mu^n_s + \mu_s \right)(dy)dz\,ds \nonumber\\
& \leq &   \frac{\kappa_1}{\lambda}  \int_0^t\int_0^{+\infty} \int_0^{+\infty} |E_n(s,z)|z^{\lambda-1}y^{\lambda}\left( \mu^n_s + \mu_s \right)(dy)dz\,ds \nonumber\\
& \leq &   \frac{\kappa_1}{\lambda} \int_0^t d_{\lambda}(\mu^n_s,\mu_s)\,M_{\lambda}(\mu^n_s +\mu_s)ds. 
\end{eqnarray}
We used (\ref{case2}).\medskip

\noindent
\underline{\textbf{Term $B_3(t)$.}}\medskip

Remark that $\left|\left(A\mathds 1_{(x,+\infty)}\right)(v,v)\right| = |\mathds 1_{(x,+\infty)}(2v) - 2\, \mathds 1_{(x,+\infty)}(v)| \leq \mathds 1_{\{v > \frac{x}{2} \}}$. \medskip

\noindent
Since $\int_0^{+\infty} \mathds 1_{\{v > \frac{x}{2}\}} x^{\lambda-1} dx = \frac{\left(2v\right)^{\lambda}}{\lambda}$, we deduce:
 \begin{eqnarray}\label{Dev:Bound_B3}
B_3(t)& \leq &\frac{1}{2n} \int_0^{+\infty} x^{\lambda-1} \int_0^t \int_0^{+\infty} K(v,v) \left|\left(A\mathds 1_{(x,+\infty)}\right)(v,v)\right|\mu^n_s(dv)\,ds dx\nonumber\\
& \leq & \frac{\kappa_0}{2\lambda\,n}\int_0^t ds\int_0^{+\infty} (2v)^{2\lambda} \mu^n_s(dv)\nonumber\\
& \leq & \frac{2^{2\lambda-1}\kappa_0}{\lambda\,n} \int_0^t M_{2\lambda}(\mu^n_s)\, ds.
\end{eqnarray}
We used (\ref{Dev:CondPhi}) and (\ref{case2}).\medskip

\noindent
\textbf{\underline{Term $B_4(t)$.}}\medskip

First, remark that from (\ref{Dev:CondPhi}) we have $\left|\varphi''_{\theta_{(x)}^n}(z)\right| \leq \frac{2}{\theta_{(x)}^n}$ for all $z$, whence, due to the Taylor-Lagrange inequality, 
\begin{eqnarray}
&& \bigg| \varphi_{\theta_{(x)}^n}\left(E_n(s,x) + \dfrac{1}{n}\left(A\mathds 1_{(x,+\infty)}\right)\left(X^i_{s},X^j_{s}\right)\right)- \varphi_{\theta_{(x)}^n}(E_n(s,x))  \hspace{3cm}\nonumber\\
&& \hspace{6cm} -\dfrac{1}{n}\left(A\mathds 1_{(x,+\infty)}\right)\left(X^i_{s},X^j_{s}\right)\varphi_{\theta_{(x)}^n}'(E_n(s,x))\bigg| \nonumber \\
&&\hspace{5cm} \leq \dfrac{2}{\theta_{(x)}^n} \left[\dfrac{1}{n}\left(A\mathds 1_{(x,+\infty)}\right)(X^i_s,X^j_s)\right]^2. \nonumber
\end{eqnarray}
Then,
\begin{eqnarray}
\mathbb E \left[ B_4(t)\right]& \leq &  \int_0^{+\infty} x^{\lambda-1} \mathbb E \Bigg[\int_0^t \int_{i<j} \int_0^{+\infty} \dfrac{2}{\theta_{(x)}^n} \left[\dfrac{1}{n}\left(A\mathds 1_{(x,+\infty)}\right)(X^i_{s-},X^j_{s-})\right]^2  \mathds 1_{\{j\leq N(s-)\}} \nonumber\\
&&\hspace{4.6cm}\mathds 1_{\left\lbrace z\leq \frac{K\left(X^i_{s-},X^j_{s-}\right)}{n} \right\rbrace } \,J(ds,d(i,j),dz)\Bigg]dx \nonumber\\
&\leq & \frac{2}{n} \int_0^t \mathbb E \left[\int_0^{+\infty} x^{\lambda-1} \sum_{i<j\leq N(s)}  \frac{K(X^i_s,X^j_s)}{n^2\,\theta_{(x)}^n}\left[\left(A\mathds 1_{(x,+\infty)}\right)(X^i_{s},X^j_{s})\right]^2  dx  \right]ds \nonumber\\
&\leq & \frac{2\kappa_0}{n} \int_0^t \mathbb E \Bigg[ \int_0^{+\infty} \dfrac{x^{\lambda-1}}{n^2\,\theta_{(x)}^n} \sum_{i<j\leq N(s)}  \left(X^i_s + X^j_s\right)^{\lambda} \nonumber \\
& & \hspace{4 cm} \left(\mathds 1_{x <X^i_s\wedge X^j_s} + \mathds 1_{ X^i_s\vee X^j_s < x <X^i_s + X^j_s} \right)  dx \Bigg]ds.\nonumber
\end{eqnarray}
We used (\ref{case2}) and (\ref{Dev:Bound_IndicCasPos}) (since the sets are disjoint, the product of indicators vanishes). Therefore, using that $(v+y)^{\lambda} < v^{\lambda} + y^{\lambda}$ and a symmetry argument, we get
\begin{equation*}
\mathbb E \left[ B_4(t)\right] \leq \frac{4\kappa_0}{n} \int_0^t \mathbb E \Bigg[\left\langle \mu^n_s(dv)\mu^n_s(dy), v^{\lambda}\int_0^{+\infty}  \frac{x^{\lambda-1}}{\theta_{(x)}^n}\left(\mathds 1_{x <v\wedge y} + \mathds 1_{ v\vee y < x <v + y} \right)  dx\right\rangle \Bigg]ds. 
\end{equation*}
According to Lemma \ref{lemma_integrals}--\textit{(\ref{lemma:3})}, and since $(v\wedge y)^{\alpha}(v\vee y)^{\beta} \leq v^{\alpha} y^{\beta} + y^{\alpha} v^{\beta}$ for $\alpha\geq0$ and $\beta\geq0$, we have
\begin{eqnarray*} 
&&\left\langle \mu^n_s(dv)\mu^n_s(dy), v^{\lambda}\int_0^{+\infty}  \frac{x^{\lambda-1}}{\theta_{(x)}^n}\left(\mathds 1_{x <v\wedge y} + \mathds 1_{ v\vee y < x <v + y} \right)  dx\right\rangle \leq \hspace{2cm}\\
&& \hspace{3cm} \frac{2\sqrt{n}}{\lambda}  \left\langle \mu^n_s(dv)\mu^n_s(dy), v^{\lambda}y^{\lambda} \right\rangle \\
&& \hspace{3cm} +  \sqrt{n}\left(2^{2\lambda+\varepsilon} + \frac{1}{\lambda}\right)\left\langle \mu^n_s(dv)\mu^n_s(dy), v^{2\lambda}y^{2\lambda+\varepsilon} +  y^{2\lambda}v^{2\lambda+\varepsilon} \right\rangle \mathds 1_{\lambda\in(0,1/2)} \\
&&  \hspace{3cm} +  \sqrt{n}\left(2^{2\lambda+\varepsilon} + \frac{1}{\lambda}\right) \left\langle \mu^n_s(dv)\mu^n_s(dy), vy^{4\lambda+\varepsilon-1} + y v^{4\lambda+\varepsilon-1}\right\rangle 1_{\lambda\in[1/2,1]}.
\end{eqnarray*} 
Finally, we deduce the bound:
\begin{eqnarray}\label{Dev:Bound_B4}
\mathbb E \left[ B_4(t)\right] & \leq &  \frac{8\kappa_0}{\lambda \sqrt{n}} \int_0^t \mathbb E \bigg[ \left[M_{\lambda}(\mu_s^n)\right]^2  + C \left[M_{2\lambda}(\mu_s^n) M_{2\lambda+\varepsilon}(\mu_s^n)\right] \mathds 1_{\lambda\in(0,1/2)}\nonumber\\
& & \hspace{2cm} + C \left[M_{1}(\mu_s^n)M_{4\lambda+\varepsilon-1}(\mu_s^n)\right]\mathds 1_{\lambda\in[1/2,1]}\bigg]ds, 
\end{eqnarray} 
where $C = \left(\lambda 2^{2\lambda+\varepsilon}+1\right)$.\medskip

\noindent
\underline{\textbf{Conclusion.}}\medskip

Gathering (\ref{Dev:Bound_B1}), (\ref{Dev:Bound_B2}), (\ref{Dev:Bound_B3}) and (\ref{Dev:Bound_B4}), from (\ref{Dev:afterIto}), we get:
\begin{eqnarray}\label{Dev:FirstUBound}
&&\mathbb E\left[ d_{\lambda}(\mu_t^n,\mu_t)\right] \hspace{3mm} \leq \hspace{3mm}  \mathbb \int_0^{+\infty} x^{\lambda-1} \varphi_{\theta_{(x)}^n}(E_n(0,x)) dx\nonumber\\
& & \hspace{6mm}+ \frac{\kappa_0(\lambda+\varepsilon)}{\lambda\varepsilon\sqrt{n}} \int_0^t \mathbb E\left[ M_{0}(\mu_s^n + \mu_s) + M_{\lambda}(\mu_s^n + \mu_s)\right]ds + \frac{\kappa_1}{\lambda} \int_0^t \mathbb E\left[  d_{\lambda}(\mu^n_s,\mu_s)\,M_{\lambda}(\mu^n_s +\mu_s)\right] ds \nonumber\\
&& \hspace{6mm} + \frac{2^{2\lambda-1}\kappa_0}{n \lambda}\int_0^t \mathbb E\left[  M_{2\lambda}(\mu^n_s) \right] ds + \frac{8 \kappa_0}{\lambda \sqrt{n}}\int_0^t \mathbb E \left[ M_{\lambda}(\mu_s^n)\right]^2 ds\nonumber\\
&& \hspace{6mm}  + \frac{8C \kappa_0}{\lambda \sqrt{n}}\int_0^t \mathbb E \Big[\left[M_{2\lambda}(\mu_s^n)\,M_{2\lambda+\varepsilon}(\mu_s^n)\right] \mathds 1_{\lambda\in(0,1/2)}+ \left[M_{1}(\mu_s^n)M_{4\lambda+\varepsilon-1}(\mu_s^n)\right]\mathds 1_{\lambda\in[1/2,1]} \Big] ds.\nonumber
\end{eqnarray}
We use (\ref{Dev:Convergence}) to bound the first term on the right-hand side. According to Proposition \ref{Prop:bounded_moments} --\textit{(\ref{Prop:bounded_moments_neg})}, $M_{\alpha}(\mu_s^n + \mu_s) \leq M_{\alpha}(\mu_0^n + \mu_0)$ a.s. for $\alpha \leq 1$. Since $\mu^n_0$ is deterministic, we get (recall that $2\lambda + \varepsilon < 1$ if $\lambda \in (0,1/2)$):
\begin{eqnarray*}
&&\mathbb E\left[ d_{\lambda}(\mu_t^n,\mu_t)\right] \leq d_{\lambda}(\mu_0^n,\mu_0)  + \frac{2}{\lambda \sqrt{n}}+ \frac{t\kappa_0(\lambda+\varepsilon)}{\lambda\varepsilon\sqrt{n}} \left( M_{0}(\mu_0^n + \mu_0)+M_{\lambda}(\mu_0^n + \mu_0)\right)\\
&& \hspace{1cm}+ \frac{\kappa_1}{\lambda} M_{\lambda}(\mu^n_0 +\mu_0)\int_0^t \mathbb E\left[  d_{\lambda}(\mu^n_s,\mu_s)\right] ds + \frac{2^{2\lambda-1}\kappa_0}{n \lambda} \int_0^t \mathbb E\left[M_{2\lambda}(\mu_s^n)\right]ds + \frac{8 t\kappa_0}{ \lambda\sqrt{n}} \left[ M_{\lambda}(\mu_0^n)\right]^2   \\
&& \hspace{1cm} + \frac{8C t\kappa_0}{\lambda\sqrt{n}} \left[M_{2\lambda}(\mu_0^n) M_{2\lambda+\varepsilon}(\mu_0^n)\right] \mathds 1_{\lambda\in(0,1/2)} \\
&& \hspace{1cm}+ \frac{8C \kappa_0}{\lambda \sqrt{n}}M_{1}(\mu_0^n)\int_0^t \mathbb E \left[ M_{4\lambda+\varepsilon-1}(\mu_s^n)\right]\mathds 1_{\lambda\in[1/2,1]} ds.
\end{eqnarray*}
Again, according to Proposition \ref{Prop:bounded_moments} --\textit{(\ref{Prop:bounded_moments_pos})}, $\mathbb E \left[ M_{\alpha}(\mu_s^n)\right] \leq M_{\alpha}(\mu_0^n) \exp[s\,C_{\lambda,\alpha} M_{\lambda}(\mu_0^n)]$ for $\alpha > 1$, and where $C_{\lambda,\alpha}$ is a positive constant depending on $\lambda$, $\alpha$ and $\kappa_0$. Thus
\begin{eqnarray*}
&&\mathbb E\left[ d_{\lambda}(\mu_t^n,\mu_t)\right] \leq d_{\lambda}(\mu_0^n,\mu_0)  + \frac{2}{\lambda \sqrt{n}}+ \frac{t\kappa_0(\lambda+\varepsilon)}{\lambda\varepsilon\sqrt{n}} \left( M_{0}(\mu_0^n + \mu_0)+M_{\lambda}(\mu_0^n + \mu_0)\right)\\
&& \hspace{1cm}+ \frac{\kappa_1}{\lambda} M_{\lambda}(\mu^n_0 +\mu_0)\int_0^t \mathbb E\left[  d_{\lambda}(\mu^n_s,\mu_s)\right] ds + \frac{2^{2\lambda-1}t\kappa_0}{n \lambda}  M_{2\lambda}(\mu_0^n) \exp[t\,C_{\lambda,\varepsilon} M_{\lambda}(\mu_0^n)]   \\
&& \hspace{1cm} + \frac{8 t\kappa_0}{\lambda\sqrt{n}} \left[ M_{\lambda}(\mu_0^n)\right]^2+ \frac{8C t\kappa_0}{\lambda\sqrt{n}} \left[M_{2\lambda}(\mu_0^n) M_{2\lambda+\varepsilon}(\mu_0^n)\right] \mathds 1_{\lambda\in(0,1/2)} \\
&& \hspace{1cm}+ \frac{8C t \kappa_0}{\lambda\sqrt{n}}M_{1}(\mu_0^n) M_{4\lambda+\varepsilon-1}(\mu_0^n)\exp[t\,C_{\lambda,\varepsilon} M_{\lambda}(\mu_0^n)] \mathds 1_{\lambda\in[1/2,1]}.
\end{eqnarray*}
Recall that $\gamma = \max\{2\lambda,4\lambda-1\}$. Observe that for $\mu\in \mathcal M^+$, $M_{\alpha}(\mu) \leq M_{0}(\mu) + M_{\beta}(\mu)$ for any $0\leq \alpha \leq \beta$. Elementary computations allow us to get:
\begin{eqnarray*}
\mathbb E\left[d_{\lambda}(\mu_t^n,\mu_t) \right] & \leq & d_{\lambda}(\mu_0^n,\mu_0)  + (1+t)\frac{C_{\lambda,\varepsilon}}{\sqrt{n}} \left( 1 + \left[M_{0}(\mu_0^n + \mu_0)\right]^2 + \left[M_{\gamma+\varepsilon}(\mu_0^n + \mu_0)\right]^2\right)  \nonumber\\
&&\times \exp\left[t\,C_{\lambda,\varepsilon} M_{\lambda}(\mu_0^n + \mu_0)\right] + C_{\lambda,\varepsilon} M_{\lambda}(\mu^n_0 +\mu_0)\int_0^t \mathbb E\left[  d_{\lambda}(\mu^n_s,\mu_s)\right] ds,
\end{eqnarray*}
for some positive constant $C_{\lambda,\varepsilon}$ depending on $\lambda$, $\varepsilon$, $\kappa_0$ and $\kappa_1$. We conclude using the Gronwall lemma that Theorem \ref{theorem} holds under (\ref{case2}).

\section{Special Case}\label{Special_Case}
\setcounter{equation}{0}
Now we are going to study the special case (\ref{special_case}) for which $\lambda\in(0,1]$. We have a better result and a simpler proof than (\ref{case2}).\medskip

In the whole section, we assume that $K$ satisfies (\ref{special_case}) for some fixed $\lambda\in(0,1]$. We fix $\varepsilon > 0$, and we assume that $\mu_0 \in \mathcal M^+_{\lambda} \cap \mathcal M^+_{2\lambda+\varepsilon}$. We denote by $(\mu_t)_{t\geq 0}$ the unique $(\mu_0,K,\lambda)$-weak solution to the Smoluchowski equation. We also consider the $(n,K,\mu_0^n)$-Marcus Lushnikov process, for some given initial condition $\mu_0^n=\frac{1}{n} \sum_{i=1}^N \delta_{x_i}$.\medskip

As we did before we introduce $E_n(t,x) = F^{\mu^ n_t}(x) - F^{\mu_t}(x)$ for $x\in (0,+\infty)$, as defined in (\ref{Intro:MetricEq1}). We observe that $\mathds 1_{(x,+\infty)} \in \mathcal H^e_{\lambda}$, since $\sup_{v>0}v^{-\lambda} |1_{(x,+\infty)}(v)|= x^{-\lambda}<+\infty$. Exactly as in Section \ref{Positive_Case} (see (\ref{Dev:EtxCasPos_2}), take the absolute value and integrate against $x^{\lambda-1}dx$), we obtain:
\begin{equation}\label{Dev:BoundGral_SpC2}
d_{\lambda}(\mu^n_t,\mu_t) \leq d_{\lambda}(\mu^n_0,\mu_0) + C_1(t) + C_2(t) + C_3(t) + C_4(t),
\end{equation}
where
\begin{eqnarray*}
C_1(t) & = &  \dfrac{1}{2} \int_0^t \int_0^{+\infty}\int_0^{+\infty} x^{\lambda-1}\bigg[ \mathds 1_{x>y} K(x-y,y) |E_n(s,x-y)| + |E_n(s,x)| K(x,y) \bigg] dx\\
& & \hspace{9.1cm}  \left( \mu^n_s + \mu_s \right)(dy) \,ds,\\
C_2(t) & = & \dfrac{1}{2}\int_0^t \int_0^{+\infty}\int_0^{+\infty}\int_0^{+\infty} x^{\lambda-1} |\partial_x K(z,y)| \left|\left(A\mathds 1_{(x,+\infty)}\right)(z,y)\right| |E_n(s,z)|\, dz\,dx  \\
& & \hspace{9.1cm}  \left( \mu^n_s + \mu_s \right)(dy) \,ds,\\
C_3(t) & = &\dfrac{1}{2n} \int_0^t \int_0^{+\infty}\int_0^{+\infty} x^{\lambda-1} K(v,v) \left|\mathds 1_{(x,+\infty)}(2v) - 2\, \mathds 1_{(x,+\infty)}(v) \right|dx \,\mu^n_s(dv)\,ds,\\
C_4(t)& = & \int_0^{+\infty} x^{\lambda-1}  \Bigg| \frac{1}{n} \int_0^t \int_{i<j} \int_0^{+\infty} (A\mathds 1_{(x,+\infty)})\left(X^i_{s-},X^j_{s-}\right)\mathds 1_{\left\lbrace z\leq \frac{K\left(X^i_{s-},X^j_{s-}\right)}{n} \right\rbrace } \mathds 1_{\{j\leq N(s-)\}} \\
& & \hspace{9.5cm}\tilde{J}(ds,d(i,j),dz)\Bigg| \, dx.
\end{eqnarray*}
We now study each term separately.\medskip

\noindent
\underline{\textbf{Term $C_1(t)$.}}\medskip

We have, using the change of variable $x \mapsto w + y$, (\ref{special_case}) and using the fact that $x^{\lambda-1}$ is a non-increasing function:
\begin{eqnarray}\label{Dev:Bound_C1}
C_1(t)&\leq & \frac{\kappa_0}{2} \int_0^t \int_0^{+\infty}\int_0^{+\infty} (w+y)^{\lambda-1} (w \wedge y)^{\lambda} |E_n(s,w)|\,dw\left( \mu^n_s + \mu_s \right)(dy)\, ds \nonumber\\
&&+\frac{\kappa_0}{2}\int_0^t \int_0^{+\infty}\int_0^{+\infty}  x^{\lambda-1}  (x \wedge y)^{\lambda} |E_n(s,x)|\,dx  \left( \mu^n_s + \mu_s \right)(dy)\,ds\nonumber  \\
&\leq & \frac{\kappa_0}{2} \int_0^t \int_0^{+\infty}\int_0^{+\infty} w^{\lambda-1} y^{\lambda} |E_n(s,w)|\,dw\left( \mu^n_s + \mu_s \right)(dy)\, ds \nonumber \\
&&+\frac{\kappa_0}{2} \int_0^t \int_0^{+\infty}\int_0^{+\infty}  x^{\lambda-1}  y^{\lambda} |E_n(s,x)|\,dx  \left( \mu^n_s + \mu_s \right)(dy)\,ds \nonumber\\
&\leq& \kappa_0\int_0^t M_{\lambda}(\mu^n_s + \mu_s)\, d_{\lambda}(\mu^n_s,\mu_s) ds. 
\end{eqnarray}

\noindent
\underline{\textbf{Term $C_2(t)$.}}\medskip

Recall (\ref{Bound_int_inf}), use (\ref{special_case}), we have immediately:
\begin{eqnarray}\label{Dev:Bound_C2}
C_2(t) & \leq &\frac{1}{\lambda}\int_0^t \int_0^{+\infty}\int_0^{+\infty}|\partial_x K(z,y)| \left(z \wedge y\right)^{\lambda} |E_n(s,z)|\,\left( \mu^n_s + \mu_s \right)(dy)\, ds \nonumber\\
& \leq & \frac{\kappa_1}{\lambda} \int_0^t \int_0^{+\infty}\int_0^{+\infty} |E_n(s,z)|z^{\lambda-1}y^{\lambda}\left( \mu^n_s + \mu_s \right)(dy)dz\,ds\nonumber\\
& = & \frac{\kappa_1}{\lambda} \int_0^t d_{\lambda}(\mu^n_s,\mu_s)\,M_{\lambda}(\mu^n_s +\mu_s)ds.
\end{eqnarray}

\noindent
\underline{\textbf{Term $C_3(t)$:}}\medskip

As before, recalling (\ref{Dev:Bound_B3}), we write:
 \begin{eqnarray}\label{Dev:Bound_C3}
C_3(t) \leq \frac{2^{2\lambda-1}\kappa_0}{\lambda\,n} \int_0^t M_{2\lambda}(\mu^n_s)\, ds.
 \end{eqnarray}

\noindent
\underline{\textbf{Term $C_4(t)$:}}\medskip

The submartingale term is going to be treated exactly as in the case $\lambda<0$. Using similar arguments as for the term $A_3(t)$, we get
\begin{eqnarray*}
\mathbb E\left[\sup_{s\in[0,t]}C_4(s)\right] &\leq &\frac{4}{\sqrt{n}} \int_0^{+\infty} x^{\lambda-1}\\ 
& & \left\lbrace  \mathbb E \left[ \int_0^t \left\langle \mu^n_s(dv)\mu^n_s(dy)\,,\,K(v,y)\,\left[\left(A\mathds 1_{(x,+\infty)}\right)(v,y)\right]^2   \right\rangle  \,ds \right] \right\rbrace ^{\frac{1}{2}} dx.
\end{eqnarray*}
Using now (\ref{Dev:Bound_IndicCasPos}) and (\ref{special_case}), we deduce that
\begin{eqnarray}\label{Dev:Bound_C4_1}
\mathbb E\left[\sup_{s\in[0,t]} C_4(s) \right] & \leq &\frac{4\sqrt{\kappa_0} }{\sqrt{n}} \int_0^{+\infty} x^{\lambda-1} \bigg\{  \mathbb E \bigg[ \int_0^t \bigg\langle \mu^n_s(dv) \mu^n_s(dy)\,,\,(v\wedge y)^{\lambda}\nonumber\\
& & \hspace{2cm} \big[\mathds 1_{\{x \in(0,v\wedge y)\}} + \,\mathds 1_{\{x \in(v \vee y , v+y)\}} \big]  \bigg\rangle  \,ds \bigg] \bigg\} ^{\frac{1}{2}} dx.
\end{eqnarray}

First assume that $x\leq 1$. Since $\mathds 1_{\{x \in(0,v\wedge y)\}} \leq \dfrac{(v\wedge y)^{\lambda}}{x^{\lambda}}$, since $\mathds 1_{\{x \in(v \vee y , v+y)\}} \leq \dfrac{(v + y)^{\lambda}}{x^{\lambda}}\leq 2^{\lambda} \dfrac{(v \vee y)^{\lambda}}{x^{\lambda}}$, and since $(v\wedge y)^{\lambda} (v\wedge y)^{\lambda} \leq v^{\lambda} y^{\lambda}$  and $(v\wedge y)^{\lambda} (v \vee y)^{\lambda} = v^{\lambda} y^{\lambda} $, we deduce that
\begin{equation*}
\left\langle \mu^n_s(dv) \mu^n_s(dy)\,,(v\wedge y)^{\lambda}\, \big[\mathds 1_{\{x \in(0,v\wedge y)\}} + \,\mathds 1_{\{x \in(v \vee y , v+y)\}} \big]  \right\rangle \leq \frac{\left(1+2^{\lambda}\right)}{x^{\lambda}} [M_{\lambda}(\mu_s^n)]^2. 
\end{equation*}
Thus, 
\begin{eqnarray}\label{Dev:Bound_C4:x_0to1}
\int_0^{1} x^{\lambda-1} \left\lbrace  \mathbb E \left[ \int_0^t \big\langle \mu^n_s(dv) \mu^n_s(dy)\,,\,(v\wedge y)^{\lambda}\, \big[\mathds 1_{\{x \in(0,v\wedge y)\}} + \,\mathds 1_{\{x \in(v \vee y , v+y)\}} \big]   \big\rangle \,ds \right] \right\rbrace ^{\frac{1}{2}} dx\nonumber\\
 \leq \,\, \sqrt{1+2^{\lambda}} \int_0^{1} x^{\frac{\lambda}{2}-1} dx\times \left\lbrace  \mathbb E \left[ \int_0^t [M_{\lambda}(\mu_s^n)]^2 \,ds \right] \right\rbrace ^{\frac{1}{2}} \hspace{1.4cm} \\
= \,\, \frac{2\sqrt{1+2^{\lambda}}}{\lambda} \left\lbrace  \mathbb E \left[ \int_0^t [M_{\lambda}(\mu_s^n)]^2 \,ds \right] \right\rbrace ^{\frac{1}{2}}. \hspace{3.3cm}\nonumber
\end{eqnarray}

Next consider $x> 1$. Since $\mathds 1_{\{x \in(0,v\wedge y)\}} \leq \dfrac{(v\wedge y)^{2\lambda+\varepsilon}}{x^{2\lambda+\varepsilon}}$, and $\mathds 1_{\{x \in(v \vee y , v+y)\}} \leq \dfrac{(v + y)^{2\lambda+\varepsilon}}{x^{2\lambda+\varepsilon}}\leq 2^{2\lambda+\varepsilon} \dfrac{(v \vee y)^{2\lambda+\varepsilon}}{x^{2\lambda+\varepsilon}}$, and since $(v\wedge y)^{\lambda} (v\wedge y)^{2\lambda+\varepsilon} \leq v^{\lambda} y^{2\lambda+\varepsilon}$  and
$(v\wedge y)^{\lambda} (v \vee y)^{2\lambda+\varepsilon} \leq v^{\lambda} y^{2\lambda+\varepsilon} + v^{2\lambda+\varepsilon} y^{\lambda}$, and using the symmetry, we deduce that
\begin{eqnarray*}
\left\langle \mu^n_s(dv) \mu^n_s(dy)\,,(v\wedge y)^{\lambda}\, \big[\mathds 1_{\{x \in(0,v\wedge y)\}} + \,\mathds 1_{\{x \in(v \vee y , v+y)\}} \big]  \right\rangle
 & \leq & \frac{\left(1+2^{2\lambda+\varepsilon+1}\right)}{x^{2\lambda+\varepsilon}} M_{\lambda}(\mu_s^n)M_{2\lambda+\varepsilon}(\mu_s^n).  
\end{eqnarray*}
Thus,
\begin{eqnarray}\label{Dev:Bound_C4:x_1toinf}
\int_1^{+\infty} x^{\lambda-1} \left\lbrace  \mathbb E \left[ \int_0^t \big\langle \mu^n_s(dv) \mu^n_s(dy)\,,\,(v\wedge y)^{\lambda}\, \big[\mathds 1_{\{x \in(0,v\wedge y)\}} + \,\mathds 1_{\{x \in(v \vee y , v+y)\}} \big]   \big\rangle \,ds \right] \right\rbrace ^{\frac{1}{2}} dx \nonumber\\
\leq\,  \sqrt{1+2^{2\lambda+\varepsilon+1}} \int_1^{+\infty} x^{-\frac{\varepsilon}{2}-1} dx\times \left\lbrace  \mathbb E \left[ \int_0^t M_{\lambda}(\mu_s^n)M_{2\lambda+\varepsilon}(\mu_s^n) \,ds \right] \right\rbrace ^{\frac{1}{2}} \hspace{5mm}\\
= \, \frac{2\sqrt{1+2^{2\lambda+\varepsilon+1}}}{\varepsilon} \left\lbrace  \mathbb E \left[ \int_0^t  M_{\lambda}(\mu_s^n)M_{2\lambda+\varepsilon}(\mu_s^n) \,ds \right] \right\rbrace ^{\frac{1}{2}}. \nonumber \hspace{3cm}
\end{eqnarray}

Gathering (\ref{Dev:Bound_C4_1}), (\ref{Dev:Bound_C4:x_0to1}) and (\ref{Dev:Bound_C4:x_1toinf}), we obtain:
\begin{eqnarray}\label{Dev:Bound_C4}
\mathbb E\left[\sup_{s\in[0,t]} C_4(s) \right] & \leq &  \frac{8\sqrt{\kappa_0} }{\sqrt{n}} \Bigg\{  \frac{\sqrt{1+2^{\lambda}}}{\lambda} \left( \mathbb E \left[ \int_0^t [M_{\lambda}(\mu_s^n)]^2 \,ds \right] \right) ^{\frac{1}{2}}\nonumber\\
& & \hspace{1.2cm}+\,\frac{\sqrt{1+2^{2\lambda+\varepsilon+1}}}{\varepsilon} \left(  \mathbb E \left[ \int_0^t  M_{\lambda}(\mu_s^n)M_{2\lambda+\varepsilon}(\mu_s^n) \,ds \right] \right) ^{\frac{1}{2}} \Bigg\}. 
\end{eqnarray}

\noindent
\underline{\textbf{Conclusion.}}\medskip

Therefore, gathering (\ref{Dev:Bound_C1}), (\ref{Dev:Bound_C2}),  (\ref{Dev:Bound_C3}) and  (\ref{Dev:Bound_C4}), we obtain:
\begin{eqnarray*}\label{Dev:pre-GronwallSpC}
\mathbb E\left[\sup_{s\in[0,t]} d_{\lambda}(\mu^n_s,\mu_s)  \right] & \leq & d_{\lambda}(\mu^n_0,\mu_0)  + \left(\kappa_0 + \frac{\kappa_1}{\lambda}\right)\int_0^t \mathbb E\left[ d_{\lambda}(\mu^n_s,\mu_s)\, M_{\lambda}(\mu^n_s + \mu_s)\right] ds \nonumber\\
& & +\, \frac{2^{2\lambda}\kappa_0}{n\,\lambda} \int_0^t \mathbb E\left[ M_{2\lambda}(\mu^n_s)\right] \, ds \\
& & +\,\frac{8\sqrt{\kappa_0} }{\sqrt{n}} \Bigg\{  \frac{\sqrt{1+2^{\lambda}}}{\lambda} \left( \mathbb E \left[ \int_0^t [M_{\lambda}(\mu_s^n)]^2 \,ds \right] \right) ^{\frac{1}{2}} \nonumber\\
& & \hspace{0.7cm}+\,\frac{\sqrt{1+2^{2\lambda+\varepsilon+1}}}{\varepsilon} \left(  \mathbb E \left[ \int_0^t M_{\lambda}(\mu_s^n)M_{2\lambda+\varepsilon}(\mu_s^n) \,ds \right] \right) ^{\frac{1}{2}} \Bigg\}. \nonumber
\end{eqnarray*} 
Observe that $M_{\alpha}(\mu_0^n) \leq M_{0}(\mu_0^n) + M_{2\lambda + \varepsilon}(\mu_0^n)$ for $\alpha = \lambda,\,2\lambda$. Proposition \ref{Prop:bounded_moments} implies that for $\alpha \in (0,1]$,  $M_{\alpha}(\mu^n_t+\mu_t) \leq M_{\alpha}(\mu^n_0+\mu_0)$ a.s. and  for $\alpha = 2\lambda,\,2\lambda+\varepsilon$, $\mathbb E \left[ M_{\alpha}(\mu_s^n)\right] \leq M_{\alpha}(\mu_0^n) \exp[s\,C_{\lambda,\alpha} M_{\lambda}(\mu_0^n)]$ where $C_{\lambda,\alpha}$ is a positive constant depending on $\lambda$, $\alpha$, $\kappa_0$ and $\kappa_1$. Since $\mu^n_0$ is deterministic, we deduce that
\begin{eqnarray*}
\mathbb E\left[\sup_{s\in[0,t]} d_{\lambda}(\mu^n_s,\mu_s)  \right]  & \leq & d_{\lambda}(\mu^n_0,\mu_0) + (1+t)\frac{C_{\lambda,\varepsilon}}{\sqrt{n}} \left( M_{0}(\mu^n_0) + M_{2\lambda+\varepsilon} (\mu^n_0)\right) \exp[t\,C_{\lambda,\varepsilon}M_{\lambda}(\mu_0^n)] \nonumber\\
& & + C_{\lambda,\varepsilon} \, M_{\lambda}(\mu^n_0+\mu_0) \int_0^t \mathbb E\left[ d_{\lambda}(\mu^n_s,\mu_s)\, \right] ds, \nonumber
\end{eqnarray*}
for some positive constant $C_{\lambda,\varepsilon}$  depending on $\lambda$, $\varepsilon$, $\kappa_0$ and $\kappa_1$. We conclude using the Gronwall lemma.

\section{Choice of the initial condition}\label{Choice_Mu0}
\setcounter{equation}{0}

The aim of this section is to prove Proposition \ref{Results:Prop}. We thus fix $\lambda\in(-\infty,1] \setminus \{0\}$ and $\mu_0\in \mathcal M^+_{\lambda} \cap M^+_{2\lambda}$. We first treat the case where $\mu_0$ is atomless, next the case where $\mu_0$ is discrete. 

\subsection{Continuum System}

We assume that $\mu_0$ is atomless. For $0<a<A<+\infty$, we consider $\mu_0|_{K}$, the restriction of $\mu_0$ to $K = [a,A]$. We consider also $N$ points $a = x_0 < x_1 <\cdots <x_N \leq A$  such that:
\begin{equation}\label{Choix:xi}
\mu_0\left([x_{i-1},\,x_i)\right)  =  \dfrac{1}{n},\,\,\,\, \forall\,\, i =1,\,\cdots, N \hspace{1cm}\textrm{ and } \hspace{1cm}
\mu_0\left([x_{N},\,A]\right)  <  \dfrac{1}{n}.
\end{equation}
We will use the points $\{x_i\}_{i=1,\cdots,\,N}$ to construct the discrete measure $\mu_0^n$ choosing $a$ and $A$ following the value of $\lambda$ as a function of $n$.

\subsubsection{Case $\lambda \in(-\infty,0)$:}
First, we choose $a_n<A_n$ as follows:
\begin{equation}\label{Choix:aA_CS_CasNeg}
a_n  =  \left(\dfrac{1}{\sqrt{n}}\right)^{\frac{1}{|\lambda |}} \hspace{1cm}\textrm{ and } \hspace{1cm}
\int_{A_n}^{+\infty} x^{\lambda} \mu_0(dx) \leq  \dfrac{1}{\sqrt{n}}. 
\end{equation}
Next, we assign the weight $\mu_0\left([x_{i-1},\,x_i)\right) = \frac{1}{n}$ to the point $x_i$ and we set
\begin{equation}\label{Choix:mu_n_CS_CasNeg}
\mu_0^n = \dfrac{1}{n} \sum_{i=1}^{N_n} \delta_{x_i}.
\end{equation}
If $\alpha \leq 0 $, we get:
\begin{eqnarray} \label{moment_CS_neg}
M_{\alpha}(\mu_0^n)& = & \frac{1}{n} \sum_{i=1}^{N_n} x_i^{\alpha} =  \sum_{i=1}^{N_n}  x_i^{\alpha}\,\mu_0([x_{i-1},\,x_i)) = \sum_{i=1}^{N_n} \int_0^{+\infty}  x_i^{\alpha}\,\mathds 1_{[x_{i-1},\,x_i)}(x) \mu_0(dx) \nonumber \\
& \leq & \sum_{i=1}^{N_n} \int_0^{+\infty} x^{\alpha}\, \mathds 1_{[x_{i-1},\,x_i)}(x)  \mu_0(dx) = \int_{a_n}^{x_{N_n}} x^{\alpha} \mu_0(dx)  \leq M_{\alpha}(\mu_0). 
\end{eqnarray}
For the distance, we have, with $K_n = [a_n,A_n]$:
\begin{eqnarray}\label{Choix:Bound_d1_CS_CasNeg}
d_ {\lambda}(\mu_{0}|_{K_n},\mu_0) & = & \int_0^{+\infty}x^{\lambda-1} \left| \int_0^{+\infty} \mathds 1_{(0,x)}(y) \left( \mu_{0}|_{K_n} - \mu_0\right)(dy) \right| dx \nonumber \\
& = &  \int_0^{+\infty}x^{\lambda-1} \left[\mu_0\left((0,x)\right) \mathds 1_{x<a_n} + \mu_0\left((A_n,x)\right)  \mathds 1_{x>A_n} + \mu_0\left((0,a_n)\right) \mathds 1_{x>a_n} \right] dx\nonumber \\
& = &   \int_0^{a_n} \int_y^{a_n}  x^{\lambda-1}\,dx\,\mu_0(dy) + \int_{A_n}^{+\infty}\int_y^{+\infty}x^{\lambda-1}\,dx\,\mu_0(dy)\nonumber  \\
& & + \int_0^{a_n}\int_{a_n}^{+\infty}x^{\lambda-1}\,dx\,\mu_0(dy)\nonumber\\
& \leq &  2 \int_0^{a_n} \int_y^{+\infty}  x^{\lambda-1}\,dx\,\mu_0(dy) + \int_{A_n}^{+\infty}\int_y^{+\infty}x^{\lambda-1}\,dx\,\mu_0(dy)\nonumber  \\
& \leq &\frac{2a_n^{|\lambda |}}{|\lambda |}  \int_0^{+\infty}y^{2\lambda} \mu_0(dy) + \frac{1}{|\lambda |}\int_{A_n}^{+\infty} y^{\lambda} \mu_0(dy) \leq  \frac{1}{|\lambda |\,\sqrt{n}} \left(2 M_{2\lambda}(\mu_0) + 1 \right), 
\end{eqnarray}
we used (\ref{Choix:aA_CS_CasNeg}) for the last inequality. Next, we introduce the notation $i_x = \max\{i : x_i \leq x;\,i=0,\cdots,\,N_n \}$ for $x>a_n$. We remark that $\mu_0^n\left((0,x]\right) =0$ if $x \leq a_n$ and $\mu_0^n\left((0,x]\right) = \mu_0\left((a_n,x_{i_x}]\right)$ if $x > a_n$. Hence,
\begin{eqnarray} \label{Choix:Bound_d2_CS_CasNeg}
d_ {\lambda}(\mu_0^n,\mu_{0}|_{K_n}) & = & \int_0^{+\infty}x^{\lambda-1} \left| \int_0^{+\infty} \mathds 1_{(0,x)}(y) \left( \mu_0^n -\mu_{0}|_{K_n}\right)(dy) \right| dx  \nonumber\\
& = &   \int_{a_n}^{A_n}x^{\lambda-1} \left|\mu_0([a_n,x_{i_x}))-\mu_0([a_n,x))\right|dx  \nonumber\\
& &+ \int_{A_n}^{+\infty}x^{\lambda-1} \left|\mu_0([a_n,x_{i_x}))-\mu_0([a_n,A_n))\right|dx\nonumber\\
& \leq &  \int_{a_n}^{A_n}x^{\lambda-1} \mu_0((x_{i_x},x))\,dx + \int_{A_n}^{+\infty} x^{\lambda-1} \mu_0([x_{N_n},A_n))\,dx\nonumber\\
& \leq & \frac{2}{n} \int_{a_n}^{+\infty} x^{\lambda-1}  dx = \frac{2 }{|\lambda | n}a_n^{\lambda} \leq \frac{2}{|\lambda | \sqrt{n}}\,.
\end{eqnarray}
We used $\left|\mu_0([a_n,x_{i_x}))-\mu_0([a_n,x))\right| = \mu_0((x_{i_x},x)) \leq \mu_0([x_{j-1},x_j))\leq \frac{1}{n}$ for some $j = 1,\cdots,\,N$, and (\ref{Choix:aA_CS_CasNeg}). Finally, from (\ref{Choix:Bound_d1_CS_CasNeg}) and (\ref{Choix:Bound_d2_CS_CasNeg}), we obtain:
\begin{equation*}
d_ {\lambda}(\mu_0^n,\mu_0) \leq d_ {\lambda}(\mu_0^n,\mu_{0}|_{K_n}) + d_ {\lambda}(\mu_{0}|_{K_n},\mu_0) \leq \frac{1}{ |\lambda|\,\sqrt{n}} \left( 2M_{2\lambda}(\mu_0) + 3\right).
\end{equation*}

\subsubsection{Case $\lambda \in(0,1]$:} 
First, we choose $a_n<A_n$ as follows:
\begin{equation}\label{Choix:aA_CS_CasPos}
\int_0^{a_n} x^{\lambda} \mu_0(dx)  \leq  \dfrac{1}{\sqrt{n}} \hspace{1cm}\textrm{ and } \hspace{1cm}
 A_n   =  \left(\sqrt{n}\right)^{\frac{1}{\lambda}}. 
\end{equation}
Next, we assign the weight $\mu_0([x_{i-1},x_i)) = \frac{1}{n}$ to the point $x_{i-1}$, recall that $x_0 = a_n$. We set
\begin{equation}\label{Choix:mu_n_CS_CasPos}
\mu^n_0(dx) = \dfrac{1}{n} \sum_{i=0}^{N_n-1} \delta_{x_i}.
\end{equation}
\noindent
If $\alpha \geq 0$, we get:
\begin{eqnarray}\label{moment_CS_pos}
M_{\alpha}(\mu_0^n)& = & \frac{1}{n} \sum_{i=0}^{N_n-1} x_i^{\alpha} =  \sum_{i=1}^{N_n}  x_{i-1}^{\alpha}\,\mu_0([x_{i-1},\,x_i)) =  \sum_{i=1}^{N_n} \int_0^{+\infty}  x_{i-1}^{\alpha}\,\mathds 1_{[x_{i-1},\,x_i)}(x) \mu_0(dx) \nonumber \\
& \leq & \sum_{i=1}^{N_n} \int_0^{+\infty} x^{\alpha}\, \mathds 1_{[x_{i-1},\,x_i)}(x)  \mu_0(dx) = \int_{a_n}^{x_{N_n}} x^{\alpha} \mu_0(dx) \leq M_{\alpha}(\mu_0). 
\end{eqnarray}
For the distance, we have, with $K_n = [a_n,A_n]$:
\begin{eqnarray}\label{Choix:Bound_d1_CS_CasPos}
d_ {\lambda}(\mu_{0}|_{K_n},\mu_0) & = & \int_0^{+\infty}x^{\lambda-1} \left| \int_0^{+\infty} \mathds 1_{[x,+\infty)}(y) \left( \mu_{0}|_{K_n} - \mu_0\right)(dy) \right| dx \nonumber \\
& = &  \int_0^{+\infty}x^{\lambda-1} \left[\mu_0\left([x,a_n)\right) \mathds 1_{x<a_n} + \mu_0\left([x,+\infty)\right)  \mathds 1_{x>A_n} + \mu_0\left([A_n,+\infty)\right) \mathds 1_{x<A_n} \right] dx\nonumber \\
& = &   \int_0^{a_n} \int_0^{y}  x^{\lambda-1}\,dx\,\mu_0(dy) + \int_{A_n}^{+\infty}\int_{A_n}^{y}x^{\lambda-1}\,dx\,\mu_0(dy)\nonumber  \\
& & + \int_{A_n}^{+\infty}\int_{0}^{A_n}x^{\lambda-1}\,dx\,\mu_0(dy)\nonumber\\
& \leq &   \int_0^{a_n} \int_0^{y}  x^{\lambda-1}\,dx\,\mu_0(dy) + 2\int_{A_n}^{+\infty}\int_0^{y}x^{\lambda-1}\,dx\,\mu_0(dy)\nonumber  \\
& \leq &  \frac{1}{\lambda} \int_0^{a_n} x^{\lambda} \mu_0(dx) +  \frac{2A_n^{-\lambda}}{\lambda}\int_0^{+\infty} y^{2\lambda} \mu_0(dy)\leq \frac{1}{\lambda\,\sqrt{n}} \left(1 + 2M_{2\lambda}(\mu_0)\right),
\end{eqnarray}
we used (\ref{Choix:aA_CS_CasPos}) for the last inequality. Next, using the notation $i_x = \min\{i : x_i \geq x;\,i=0,\cdots,\,N-1 \}$ for $x>a_n$,  we remark that $\mu_0^n\left([x,+\infty)\right) = 0$ if $x \geq A_n$ and $\mu_0^n\left([x,+\infty)\right) = \mu_0\left([x_{i_x},A_n)\right)$ if $x < A_n$. Hence,
\begin{eqnarray} \label{Choix:Bound_d2_CS_CasPos}
d_ {\lambda}(\mu_0^n,\mu_{0}|_{K_n}) & = & \int_0^{+\infty}x^{\lambda-1} \left| \int_0^{+\infty} \mathds 1_{(x,+\infty)}(y) \left( \mu_0^n -\mu_{0}|_{K_n}\right)(dy) \right| dx  \nonumber\\
& = &   \int_{a_n}^{A_n}x^{\lambda-1} \left|\mu_0((x_{i_x},A_n))-\mu_0((x,A_n))\right|dx \nonumber\\
& & + \int_0^{a_n} x^{\lambda-1} \left|\mu_0([x_{i_x},A_n))-\mu_0((a_n,A_n))\right|dx\nonumber\\
& = &  \int_{a_n}^{A_n}x^{\lambda-1} \mu_0((x,x_{i_x}))dx \,\leq \, \frac{1}{n} \int_{0}^{A_n} x^{\lambda-1}  dx = \frac{1}{\lambda\, n} A_n^{\lambda } \leq \frac{1}{\lambda \,\sqrt{n}},
\end{eqnarray}
we used $ \left|\mu_0((x_{i_x},A_n))-\mu_0((x,A_n))\right| = \mu_0((x,x_{i_x})) \leq \mu_0([x_{j-1},x_j))$ for some $j = 1,\cdots,\,N$, and (\ref{Choix:aA_CS_CasPos}). Finally, from (\ref{Choix:Bound_d1_CS_CasPos}) and (\ref{Choix:Bound_d2_CS_CasPos}), we deduce:
\begin{eqnarray*}
d_ {\lambda}(\mu_0^n,\mu_0)  & \leq & \frac{2}{\lambda\, \sqrt{n}} \left(M_{2\lambda}(\mu_0) + 1\right). 
\end{eqnarray*}

\subsection{Discrete System}

Let us thus, consider $\mu_0 \in \mathcal M^+$ with support in $\mathbb N$, i.e.
\begin{equation}
\mu_0 = \sum_{k\geq 1} \alpha_k\, \delta_{k}, \hspace{5mm} \textrm{with } \alpha_k \in \mathbb R_+.
\end{equation}
We set for $A\in\mathbb N$:
\begin{equation}\label{Choix:muA_DS}
\mu_0^A = \sum_{k = 1}^A \alpha_k\, \delta_{k}.
\end{equation}

\subsubsection{Case $\lambda \in(-\infty, 0)$:}
We choose $A_n$ such that:
\begin{equation}\label{Choix:An_DS_CasNeg}
\sum_{k > A_n} \alpha_k k^{\lambda} \leq \frac{1}{\sqrt{n}}, 
\end{equation}
and we set,
\begin{equation}\label{Choix:mun_DS}
\mu_{0}^{n} = \frac{1}{n}\sum_{k = 1}^{A_n} \alpha_k^n\, \delta_{k}, \hspace{5mm} \textrm{with}
\end{equation}
\begin{equation}\label{Choix:alphak_DS_CasNeg}
\left\{
\begin{array}{lcl}
\alpha_1^n & = & \lfloor n \alpha_1 \rfloor, \\
\alpha_k^n & = & \lfloor n (\alpha_1 + \cdots + \alpha_k )\rfloor -  \lfloor n (\alpha_1 + \cdots + \alpha_{k-1} )\rfloor \textrm{ for } k = 2,\cdots,\,A_n, 
\end{array}\right.
\end{equation}
where $\lfloor \cdot \rfloor$ is the floor function. Remark that chosen in this way, the $\alpha_k^n$ are non-negative integers and $\mu_0^n$ can be written as $\frac{1}{n} \sum_{i=1}^{N_n}\delta_{x_i}$, hence $\mu_0^n$ is the measure we search. Observe that for $k = 1,\cdots,\, A_n$, we have 
\begin{eqnarray}\label{Choix:condalpha_DS_CasNeg} 
\left|\sum_{i=1}^k \left(\frac{1}{n}\alpha_i^n - \alpha_i\right)\right| & = & \left| \frac{1}{n} \left(\alpha_1^n+\cdots+\alpha_k^n \right) - \left(\alpha_1+\cdots+\alpha_k \right) \right| \nonumber\\
& = & \left| \frac{1}{n} \left\lfloor n \left(\alpha_1+\cdots+\alpha_k \right) \right\rfloor - \left(\alpha_1+\cdots+\alpha_k \right) \right|   \,\leq \,\frac{1}{n}.\end{eqnarray}
If $\alpha \leq 0$, we have:
\begin{eqnarray}\label{moment_DS_neg}
M_{\alpha}\left(\mu_{0}^{n}\right)& = & \frac{1}{n} \sum_{k=1}^{A_n} \alpha_k^n k^{\alpha}  \nonumber\\
& = & \frac{1}{n}\lfloor n \alpha_1 \rfloor + \frac{1}{n} \sum_{k=2}^{A_n}  \lfloor n (\alpha_1 + \cdots + \alpha_k )\rfloor k^{\alpha} - \frac{1}{n} \sum_{k=2}^{A_n}  \lfloor n (\alpha_1 + \cdots + \alpha_{k-1})\rfloor k^{\alpha}  \nonumber\\
& = & \frac{1}{n}\lfloor n \alpha_1 \rfloor + \frac{1}{n} \sum_{k=2}^{A_n}  \lfloor n (\alpha_1 + \cdots + \alpha_k )\rfloor k^{\alpha} - \frac{1}{n} \sum_{k=1}^{A_n-1}  \lfloor n (\alpha_1 + \cdots + \alpha_{k})\rfloor (k+1)^{\alpha}  \nonumber\\
& = & \frac{1}{n} \left( \lfloor n \alpha_1 \rfloor + A_n^{\alpha} \lfloor n (\alpha_1 + \cdots + \alpha_{A_n})\rfloor - 2^{\alpha}\lfloor n \alpha_1 \rfloor  \right) \nonumber\\
&& + \frac{1}{n} \sum_{k=2}^{A_n-1}  \lfloor n (\alpha_1 + \cdots + \alpha_{k})\rfloor \left(k^{\alpha} - (k+1)^{\alpha} \right)  \nonumber\\
& \leq & \alpha_1 \left(1 - 2^{\alpha}\right) + A_n^{\alpha} \left(\alpha_1 + \cdots + \alpha_{A_n}\right) + \sum_{k=2}^{A_n-1}  (\alpha_1 + \cdots + \alpha_{k})  \left(k^{\alpha} - (k+1)^{\alpha} \right)  \nonumber\\
& = & \sum_{k=1}^{A_n-1} \alpha_{k} \left[A_n^{\alpha} + \sum_{j=k}^{A_n-1}\left(j^{\alpha} - (j+1)^{\alpha} \right) \right] + A_n^{\alpha} \alpha_{A_n} = \sum_{k=1}^{A_n} \alpha_{k} k^{\alpha} \leq M_{\alpha}(\mu_0).
\end{eqnarray}
Next, for the distance, we have:
\begin{eqnarray} \label{Choix:d1_DS_CasNeg}  
d_{\lambda}\left(\mu_{0}^{A_n},\mu_0\right)  & \leq & \int_0^{+\infty}x^{\lambda-1} \left| \int_0^{+\infty} \mathds 1_{(0,x)}(y)\left(\mu_{0}^{A_n} - \mu_0\right) (dy) \right| dx \nonumber \\
& = &  \int_0^{+\infty}x^{\lambda-1} \int_0^{x} \sum_{k>A_n} \alpha_k\,\delta_k(dy) dx \, = \,\sum_{k>A_n} \alpha_k \int_{k}^{+\infty} x^{\lambda-1}dx  \nonumber\\
& = & \frac{1}{|\lambda |} \sum_{k> A_n} \alpha_k\,k^{\lambda}\, \leq\, \frac{1}{|\lambda | \sqrt{n}},
\end{eqnarray}
we used (\ref{Choix:An_DS_CasNeg}). Next,
\begin{eqnarray} \label{Choix:Bound_d2_DS_CasNeg}
d_ {\lambda}\left(\mu_{0}^{n} , \mu_0^{A_n} \right)  & = & \int_0^{+\infty}x^{\lambda-1} \left| \int_0^{+\infty} \mathds 1_{(0,x)}(y)\left(\mu_{0}^{n} - \mu_0^{A_n} \right) (dy) \right| dx \nonumber \\
& = &  \sum_{k=1}^{A_n-1}\int_k^{k+1}x^{\lambda-1} \left|\sum_{i=1}^k \left(\frac{1}{n}\alpha_i^n - \alpha_i\right)\right| dx+ \int_{A_n}^{+\infty}x^{\lambda-1} \left|\sum_{i=1}^{A_n} \left(\frac{1}{n}\alpha_i^n - \alpha_i\right)\right| dx\nonumber\\
 & \leq & \frac{2}{n}\int_1^{+\infty}x^{\lambda-1} dx \leq \frac{2}{|\lambda | n},
\end{eqnarray}
we used (\ref{Choix:condalpha_DS_CasNeg}) for the last inequality. Finally, from (\ref{Choix:d1_DS_CasNeg}) and (\ref{Choix:Bound_d2_DS_CasNeg}), we have:
\begin{equation*}
d_ {\lambda}\left(\mu_{0}^{n},\mu_0\right)  \leq d_ {\lambda}\left(\mu_{0}^{n},\mu_0^{A_n}\right) + d_ {\lambda}\left(\mu_0^{A_n},\mu_0\right) \leq \frac{1}{|\lambda |\, \sqrt{n}} \left(1+\frac{2}{\sqrt{n}}\right).
\end{equation*}

\subsubsection{Case $\lambda \in(0,1]$:}
We set
\begin{equation}\label{Choix:An_DS_CasPos1}
A_n  =  \left\lfloor \left(\sqrt{n}\,\right)^{\frac{1}{\lambda}}\right\rfloor + 1,
\end{equation}
Note that chosen in this way, we have $A_n^{-\lambda} \leq \frac{1}{\sqrt{n}}$, implying
\begin{equation}\label{Choix:An_DS_CasPos2}
\sum_{k\geq A_n} \alpha_k\,k^{\lambda} \leq A_n^{-\lambda}\sum_{k\geq A_n} \alpha_k\,k^{2\lambda} \leq \frac{1}{\sqrt{n}} M_{2\lambda}(\mu_0).
\end{equation}
We set the measure $\mu_{0}^{n}$ as defined in (\ref{Choix:mun_DS}), with
\begin{equation}\label{Choix:alphak_DS_CasPos}
\alpha_k^n = \left\lfloor n \sum_{i\geq k} \alpha_i \right\rfloor - \left\lfloor n \sum_{i\geq k+1} \alpha_i \right\rfloor, \,\,\,\,\, \textrm{ for } k=1,\cdots,\, A_n.
\end{equation}
Observe that, since $\sum_{k\geq 1} \alpha_k = M_{0}(\mu_0)  \leq M_{\lambda }(\mu_0) = \sum_{k\geq 1} \alpha_k k^{\lambda } < +\infty$, the weights $\{\alpha_k^n\}_{k \geq 1}$ are well-defined. Remark that chosen in this way, the $\alpha_k^n$ are non-negative integers and $\mu_0^n$ can be written as $\frac{1}{n}\sum_{i=1}^{N_n}\delta_{x_i}$, hence $\mu_0^n$ is the measure we search. \medskip

\noindent
For $1\leq j\leq A_n$, we have:
\begin{eqnarray}\label{Choix:condalpha_DS_CasPos} 
\left|\sum_{k=j}^{A_n} \left(\frac{1}{n}\alpha_k^n - \alpha_k\right)\right| & = &  \left| \frac{1}{n}  \left\lfloor n \sum_{i\geq j} \alpha_i \right\rfloor - \frac{1}{n} \left\lfloor n \sum_{i\geq A_n+1} \alpha_i \right\rfloor - \sum_{k = j}^{A_n} \alpha_k \right|\nonumber \\
& \leq &  \left| \frac{1}{n}  \left\lfloor n \sum_{i = j}^{A_n} \alpha_i \right\rfloor +  \frac{1}{n} - \sum_{k = j}^{A_n} \alpha_k \right| \nonumber\\
& \leq & \frac{1}{n}  +  \left| \frac{1}{n}  \left\lfloor n \sum_{i = j}^{A_n} \alpha_i \right\rfloor - \sum_{k = j}^{A_n} \alpha_k \right|  \leq  \frac{2}{n}.
\end{eqnarray}

\noindent
If $\alpha\geq 0$, we have:
\begin{eqnarray}\label{moment_DS_pos}
M_{\alpha}(\mu_{0}^{n})& = & \frac{1}{n} \sum_{k=1}^{A_n} \alpha_k^n k^{\alpha}  = \frac{1}{n} \sum_{k=1}^{A_n}  \left\lfloor n \sum_{i\geq k} \alpha_i \right\rfloor k^{\alpha} - \frac{1}{n} \sum_{k=1}^{A_n}  \left\lfloor n \sum_{i\geq k+1} \alpha_i \right\rfloor k^{\alpha} \nonumber \\
& = & \frac{1}{n}\sum_{k=2}^{A_n} \left\lfloor n \sum_{i \geq k }   \alpha_i \right\rfloor \left[k^{\alpha} - (k-1)^{\alpha} \right] +  \frac{1}{n}  \left\lfloor n \sum_{i \geq 1 } \alpha_i \right\rfloor - \frac{A_n^{\alpha}}{n}  \left\lfloor n \sum_{i \geq A_n+1} \alpha_i \right\rfloor \nonumber \\
 & \leq & \sum_{k\geq 1} \left( \sum_{i \geq k }   \alpha_i \right) \left[k^{\alpha} - (k-1)^{\alpha} \right] \, = \,  \sum_{i \geq 1 } \alpha_i \sum_{k = 1}^i \left[k^{\alpha} - (k-1)^{\alpha} \right]\, = \, M_{\alpha}(\mu_0),
\end{eqnarray} 
For the distance, we have :
\begin{eqnarray} \label{Choix:d1_DS_CasPos}
d_{\lambda}\left(\mu_{0}^{A_n} , \mu_0\right)  & = & \int_0^{+\infty}x^{\lambda-1} \left| \int_0^{+\infty} \mathds 1_{[x,+\infty)}(y) \left(\mu_{0}^{A_n} - \mu_0\right) (dy) \right| dx \nonumber \\
& = &  \int_0^{+\infty}x^{\lambda-1} \int_x^{+\infty} \sum_{k>A_n} \alpha_k\,\delta_k(dy) dx \, = \,\sum_{k>A_n} \alpha_k \int_0^{k} x^{\lambda-1}dx  \nonumber\\  
& = & \frac{1}{\lambda } \sum_{k> A_n} \alpha_k\,k^{\lambda}\, \leq\, \frac{1}{\lambda \sqrt{n}}M_{2\lambda}(\mu_0),
\end{eqnarray}
we used (\ref{Choix:An_DS_CasPos2}). Next,
\begin{eqnarray} \label{Choix:Bound_d2_DS_CasPos}
d_ {\lambda}\left(\mu_{0}^{n} , \mu_0^{A_n} \right) & = & \int_0^{+\infty}x^{\lambda-1} \left| \int_0^{+\infty} \mathds 1_{[x,+\infty)}(y) \left(\mu_{0}^{n} - \mu_0^{A_n} \right) (dy) \right| dx\nonumber  \\
& = &  \sum_{j=1}^{A_n}\int_{j-1}^j x^{\lambda-1} \left|\sum_{k=j}^{A_n} \left(\frac{1}{n}\alpha_k^n - \alpha_k\right)\right| dx\nonumber\\
& \leq & \frac{2}{n}\int_0^{A_n}x^{\lambda-1} dx\, \leq \, \frac{2A_n^{\lambda}}{\lambda\, n}\, \leq \, \frac{4}{\lambda\, \sqrt{n}},
\end{eqnarray}
we used (\ref{Choix:condalpha_DS_CasPos}) and (\ref{Choix:An_DS_CasPos1}). Finally, from (\ref{Choix:d1_DS_CasPos}) and (\ref{Choix:Bound_d2_DS_CasPos}), we obtain:
\begin{eqnarray*}
d_ {\lambda}(\mu_{0}^{n},\mu_0)  & \leq & d_ {\lambda}(\mu_{0}^{n},\mu_0^{A_n}) + d_ {\lambda}(\mu_0^{A_n},\mu_0)\, \leq \,  \frac{1}{\lambda \, \sqrt{n}} \left(M_{2\lambda}(\mu_0)+4\right).
\end{eqnarray*}

\subsection{Conclusion}
In any case, ($\lambda \in (-\infty,1] \setminus \{0\}$ and $\mu_0$ either atomless or with support in $\mathbb N$), we have built a measure of the form $\mu_0^n = \frac{1}{n}\sum_{i=1}^{N_n} \delta_{x_i}$ satisfying the desired conditions on the moments and distance. It is straightforward to show that $N_n = n\,\left\langle\mu_0^n(dx), 1 \right\rangle$. Hence, according to (\ref{moment_CS_neg}), (\ref{moment_CS_pos}), (\ref{moment_DS_neg}) and (\ref{moment_DS_pos}), we deduce,
\begin{equation}
N_n = n M_0(\mu_0^n) \leq n M_0(\mu_0).
\end{equation}
This concludes the proof of Proposition \ref{Results:Prop}.

\begin{appendix} 
\section{}\label{Preliminaries}
\setcounter{equation}{0}

This section is devoted to some technical issues.

\begin{lemma}\label{Intro:lemmaWPdnss}
Consider $\lambda\in(-\infty,1]\setminus\{0\}$. Then, there exists a positive constant $C_{\phi}$ depending on $\phi$ and $\lambda$ such that
\begin{equation}\label{Intro:WellPdnss}
\left\{
\begin{array}{ll}
\textrm{if }\lambda\in(-\infty,0), & (x+y)^{\lambda}\left| (A\phi)(x,y)\right|\leq C_{\phi} (xy)^{\lambda}\hspace{5mm} \forall \phi \in \mathcal H_{\lambda},\\
\textrm{if }\lambda\in(0,1], & (x+y)^{\lambda}\left| (A\phi)(x,y)\right|\leq C_{\phi}\left(1 + x^{2\lambda}+y^{2\lambda}\right)\hspace{5mm} \forall \phi \in \mathcal H_{\lambda},\\
\textrm{if }\lambda\in(0,1], & (x\wedge y)^{\lambda}\,\left| (A\phi)(x,y)\right|\leq C_{\phi} (xy)^{\lambda} \hspace{5mm} \forall \phi \in \mathcal H^e_{\lambda}.\\
\end{array}\right.
\end{equation}
\end{lemma}
\begin{proof}
Assume first that $\lambda\in(-\infty,0)$ and $\phi \in \mathcal H_{\lambda}$. Since $|\phi(x)|\leq Cx^{\lambda}$ for some constant $C>0$, we have
\begin{equation*}
(x+y)^{\lambda}\left| (A\phi)(x,y)\right|\leq C (x^{\lambda}\wedge y^{\lambda}) \left[(x + y)^{\lambda} + x^{\lambda} + y^{\lambda} \right] \leq  C (x y)^{\lambda}.
\end{equation*}
Next, for $\lambda\in(0,1]$ and $\phi \in \mathcal H_{\lambda}$, since $|\phi(x)|\leq C (1 + x^{\lambda})$ for some constant $C>0$, we have
\begin{equation*}
(x+y)^{\lambda}\left| (A\phi)(x,y)\right|  \leq   C (x + y)^{\lambda}  \left [3 + (x + y)^{\lambda} + x^{\lambda} + y^{\lambda} \right] \leq C \left(1 + x^{2\lambda}+y^{2\lambda}\right) .
\end{equation*}
Finally, for $\lambda\in(0,1]$ and $\phi(x)\in \mathcal H^e_{\lambda}$, there exists $C>0$ such that $|\phi(x)|\leq C x^{\lambda}$ and we have
\begin{equation*}
(x\wedge y)^{\lambda}\left| (A\phi)(x,y)\right|\leq C (x \wedge y)^{\lambda} \left[(x + y)^{\lambda} + x^{\lambda} + y^{\lambda} \right] \leq C (xy)^{\lambda}.
\end{equation*}
\end{proof}

\begin{lemma}\label{Dev:lemmaforfubini}
Let $\lambda\in(-\infty,1]\setminus \{0\}$  and $K \in W^{1,\infty}\left((\varepsilon,1/\varepsilon)^2\right)$ for every $\varepsilon \in (0,1)$. If $K$ satisfies (\ref{case1}), then for all $(x,v,y)\in (0,+\infty)^3$: 
\begin{eqnarray}\label{Dev:WeakDerCasNeg}
&& K(v,y) \left[ \mathds 1_{(0,x]}(v+y)- \mathds 1_{(0,x]}(v) -  \mathds 1_{(0,x]}(y) \right]\nonumber\\ 
&&\hspace{3.7cm}=\,K(x-y,y) \mathds 1_{(0,x]}(v+y) -  K(x,y)\mathds 1_{(0,x]}(v)\\
&&\hspace{4.2cm}-\displaystyle\int_v^{+\infty}\partial_x K(z,y)\left[ \mathds 1_{(0,x]}(z+y)- \mathds 1_{(0,x]}(z) -  \mathds 1_{(0,x]}(y) \right] dz. \nonumber
\end{eqnarray}
If $K$ satisfies (\ref{case2}) or (\ref{special_case}), then for all $(x,v,y)\in (0,+\infty)^3$: 
\begin{eqnarray}\label{Dev:WeakDerCasPos}
&&K(v,y) \left[ \mathds 1_{(x,+\infty)}(v+y)- \mathds 1_{(x,+\infty)}(v) -  \mathds 1_{(x,+\infty)}(y) \right]\nonumber\\ 
&&\hspace{2.7cm}=\,K(x-y,y)\mathds 1_{x > y}\mathds 1_{(x,+\infty)}(v+y) -  K(x,y)\mathds 1_{(x,+\infty)}(v)\\
&&\hspace{3.2cm}+\displaystyle\int_0^{v}\partial_x K(z,y)\left[ \mathds 1_{(x,+\infty)}(z+y)- \mathds 1_{(x,+\infty)}(z) -  \mathds 1_{(x,+\infty)}(y) \right] dz.\nonumber
\end{eqnarray}
\end{lemma}
\begin{proof}
For $\lambda\in(-\infty,1]\setminus \{0\}$ we have that $K(\cdot,\cdot)$ and its weak partial derivatives belong to $L^{\infty}\left((\varepsilon,1/\varepsilon)^2\right)$, whence, for all $0 < a \leq b < +\infty$ and for all $y > 0$ (see for exemple \cite{Ziemer}): 
\begin{equation}
\int_a^b \partial_x K(z,y) dz = K(b,y) - K(a,y).
\end{equation}
First assume (\ref{case1}), and fix $\lambda \in (-\infty,0)$. Remark that:
\[ \int_a^{+\infty} \partial_x K(z,y) dz = \lim_{b\rightarrow +\infty} \int_a^b \partial_x K(z,y) dz= \lim_{b\rightarrow +\infty} K(b,y) - K(a,y) = - K(a,y).\]
Hence,
\begin{eqnarray*}
\int_v^{+\infty}\partial_x K(z,y) \mathds 1_{(0,x]}(z+y) dz & = & \mathds 1_{x>y} \mathds 1_{v\leq x-y}\int_0^{+\infty}\partial_x K(z,y)\mathds 1_{v \leq z \leq x-y}\, dz\\
& = & \mathds 1_{(0,x]}(v+y) \left[ K(x-y,y) - K(v,y) \right]. 
\end{eqnarray*}
Next,
\begin{eqnarray*}
-\int_v^{+\infty}\partial_x K(z,y) \mathds 1_{(0,x]}(z)  dz & = &- \mathds 1_{v\leq x}\int_0^{+\infty}\partial_x K(z,y)\mathds 1_{v\leq z \leq x}\, dz\\
& = & \mathds 1_{(0,x]}(v) \left[ K(v,y) - K(x,y) \right], 
\end{eqnarray*}
\begin{eqnarray*}
-\int_v^{+\infty}\partial_x K(z,y) \mathds 1_{(0,x]}(y)  dz & = &- \mathds 1_{y\leq x}\int_0^{+\infty}\partial_x K(z,y)\mathds 1_{v\leq z }\, dz\\
& = & \mathds 1_{(0,x]}(y) K(v,y). 
\end{eqnarray*}
Adding these three terms to the terms on the right-hand of (\ref{Dev:WeakDerCasNeg}) the result follows. \medskip

Next, assume (\ref{case2}) or (\ref{special_case}). Observe that for $(x,y,z)\in (0,+\infty)^3$, we have
\begin{equation}\label{Results:fubinidem}
\mathds 1_{z>x-y} - \mathds 1_{y>x} = \mathds 1_{y \leq x} \mathds 1_{z>x-y}.
\end{equation}
Thus,
\begin{eqnarray*}
\int_0^{v}\partial_x K(z,y)\left[ \mathds 1_{z+y>x}- \mathds 1_{z>x} -  \mathds 1_{y>x}\right] dz \hspace{5cm}\\
=\, \int_0^{v}\partial_x K(z,y)\left[  \mathds 1_{y \leq x} \mathds 1_{z>x-y}   - \mathds 1_{z>x} \right]dz\hspace{4.2cm}\\
=\,  \mathds 1_{y \leq x} \mathds 1_{v>x-y} \int_{x-y}^{v}\partial_x K(z,y)dz -\mathds 1_{v > x} \int_{x}^{v}\partial_x K(z,y)dz\hspace{1.4cm}\\
=\, \mathds 1_{y \leq x} \mathds 1_{v>x-y} \left[K(v,y) - K(x-y,y) \right] - \mathds 1_{v > x} \left[K(v,y) - K(x,y) \right]\\
=\, \left[\mathds 1_{v>x-y} - \mathds 1_{y>x} \right] K(v,y) -  \mathds 1_{y < x} \mathds 1_{v>x-y}K(x-y,y)\hspace{1.85cm}\\
 - \mathds 1_{v > x} \left[K(v,y) - K(x,y) \right].\hspace{1mm}
\end{eqnarray*}
Adding these terms to the remaining terms on the right-hand of (\ref{Dev:WeakDerCasPos}),  the result follows. 
\end{proof}
Now we will show a lemma which is useful to show Proposition \ref{Prop:bounded_moments} stating that the $\alpha$-moments of $\mu_0$ and $\mu^n_0$ remain bounded in time.
\begin{lemma}\label{Prelim:bound_alphageq1}
Consider $\alpha \in \mathbb R$, $\lambda\in(-\infty,1]$ and a kernel $K$ satisfying either (\ref{case1}), (\ref{case2}) or (\ref{special_case}). We set $\vartheta(x) = x^{\alpha}$. Then, 
\begin{enumerate}[(i)]
\item \label{Prelim:decreasing} if $\alpha \in (-\infty,1]$,  $(A\vartheta)(x,y) \leq 0$, for $(x,y)\in(0,+\infty)^2$,
\item \label{bound_alphageq1} if $\alpha \in (1,+\infty)$, $K(x,y)\, |(A\vartheta)(x,y)| \leq C_{\lambda,\alpha}\left(x^{\alpha}y^{\lambda} + x^{\lambda}y^{\alpha}\right)$, for $(x,y)\in(0,+\infty)^2$,
\end{enumerate}
where $C_{\lambda,\alpha}$ is a positive constant depending on $\lambda$, $\alpha$ and $\kappa_0$.
\end{lemma}
\begin{proof} 
Point \textit{(\ref{Prelim:decreasing})} is obvious, since for $\alpha \leq 1$, $ (x+y)^{\alpha} - x^{\alpha}  - y^{\alpha}  \leq  (x^{\alpha} + y^{\alpha}) - x^{\alpha}  - y^{\alpha} = 0$.\medskip

Next, if $\alpha > 1$, using (\ref{case1}), (\ref{case2}) or (\ref{special_case}), there holds $K(x,y)\leq \kappa_0 (x^{\lambda}+y^{\lambda})$. We get
\begin{eqnarray*}
K(x,y) |(A\vartheta)(x,y)|& \leq &\kappa_0 x^\lambda \left[\left|(x+y)^{\alpha} - x^{\alpha}\right| + y^{\alpha}\right] + \kappa_0 y^\lambda \left[\left|(x+y)^{\alpha} - y^{\alpha}\right| + x^{\alpha}\right]\\
&\leq & \alpha \kappa_0 \left[ \left(x^{\lambda}y^{\alpha} + x^{\alpha} y^{\lambda}\right) + (x+y)^{\alpha-1} \left(x^{\lambda}y + x y^{\lambda}\right)\right]\\
& \leq & C \left[ \left(x^{\lambda}y^{\alpha} + x^{\alpha} y^{\lambda}\right) + \left(x^{\alpha-1}+y^{\alpha-1}\right)\left(x^{\lambda}y + x y^{\lambda}\right)\right]\\
& \leq &  C  \left(x^{\lambda}y^{\alpha} + x^{\alpha} y^{\lambda} + x^{\lambda+\alpha-1}y + x y^{\lambda+\alpha-1}\right).
\end{eqnarray*}
Note that $x^{\lambda+\alpha-1}y = x^{\alpha}y^{\lambda} \left( \frac{y}{x}\right)^{1-\lambda} = x^{\lambda}y^{\alpha} \left( \frac{x}{y}\right)^{\alpha-1}\leq x^{\alpha} y^{\lambda}\mathds 1_{x>y} + x^{\lambda}y^{\alpha}\mathds 1_{x\leq y}$. We have an equivalent bound for the fourth term and the result follows.
\end{proof}

\begin{proposition}\label{Prop:bounded_moments}
Consider $\lambda \in (-\infty,1]\setminus\{0\}$ and a coagulation kernel $K$ satisfying either (\ref{case1}), (\ref{case2}) or (\ref{special_case}). Let $\mu_0 \in \mathcal M^{+}_{\lambda}$, and denote by $(\mu_t)_{t\in[0,T)}$ the $(\mu_{0},K,\lambda)$-weak solution to Smoluchowski's equation. Let $\mu^n_0$ be a deterministic discrete measure and  $(\mu^n_t)_{t\geq 0}$ the associated $(n,K,\mu^n_0)$-Marcus-Lushnikov process. Let $\alpha \in \mathbb R$, then
\begin{enumerate}[(a)]
\item \label{Prop:bounded_moments_neg} if $\alpha \leq 1$, $t\mapsto M_{\alpha}(\mu_t)$ and $t\mapsto M_{\alpha}(\mu^n_t)$ are a.s. non-increasing;

\item \label{Prop:bounded_moments_pos} if $\alpha > 1$, there exists a positive constant $C_{\lambda,\alpha}$ depending on $\lambda$, $\alpha$ and $\kappa_0$ such that $ M_{\alpha}(\mu_t) \leq  M_{\alpha}(\mu_0)\exp\left[t\,C_{\lambda,\alpha} M_{\lambda}(\mu_0)\right]$ and $\mathbb E \left[M_{\alpha}(\mu^n_t)\right] \leq  M_{\alpha}(\mu^n_0)\exp\left[t\,C_{\lambda,\alpha}M_{\lambda}(\mu^n_0)\right]$.
\end{enumerate}
\end{proposition}

\begin{proof}
Let $\phi(x) = x^{\alpha}$. For point \textit{(\ref{Prop:bounded_moments_neg})}, first consider  (\ref{Intro:SmoEq_Weak}). From  Lemma \ref{Prelim:bound_alphageq1}.--\textit{(\ref{Prelim:decreasing})}, we immediately deduce
\begin{eqnarray*}
\dfrac{d}{dt}\langle \mu_t(dx), \phi(x) \rangle =  \dfrac{d}{dt} M_{\alpha}(\mu_t) = \dfrac{1}{2} \left\langle\mu_t(dx) \mu_t(dy), \left(A\phi\right)(x,y)K(x,y)\right\rangle  \leq 0.
\end{eqnarray*}
Next, consider (\ref{Dev:Eq-mu_n-1}) and remark that $ \phi\left(X^i_{s-}+ X^j_{s-}\right) - \phi\left(X^i_{s-}\right) - \phi\left(X^j_{s-}\right) = \left(A\phi\right) \left(X^i_{s-},X^j_{s-}\right)$. From Lemma \ref{Prelim:bound_alphageq1}.--\textit{(\ref{Prelim:decreasing})}  and since $J$ is a positive measure, we deduce that the jumps of $M_{\alpha}(\mu^n_t) = \langle \mu^n_t(dx), \phi(x) \rangle$ are negative and the conclusion follows.\medskip

For point \textit{(\ref{Prop:bounded_moments_pos})}, consider (\ref{Intro:SmoEq_Weak}). According to Lemma \ref{Prelim:bound_alphageq1}.--\textit{(\ref{bound_alphageq1})}, we deduce:
\begin{eqnarray*}
\dfrac{d}{dt} M_{\alpha}\left(\mu_t\right)\, = \,  \dfrac{1}{2} \left\langle\mu_t(dx) \mu_t(dy), (A\phi)(x,y)K(x,y)\right\rangle  & \leq & \dfrac{C_{\lambda,\alpha}}{2} \left\langle \mu_t(dx)  \mu_t (dy), x^{\alpha}y^{\lambda} + x^{\lambda}y^{\alpha}\right\rangle \\
& \leq & C_{\lambda,\alpha} \,M_{\lambda}\left(\mu_t\right) M_{\alpha}\left(\mu_t\right)\\
& \leq & C_{\lambda,\alpha} \, M_{\lambda}\left(\mu_0\right)M_{\alpha}\left(\mu_t\right),
\end{eqnarray*}
we used the point \textit{(\ref{Prop:bounded_moments_neg})}. We conclude using the Gronwall lemma.\medskip

\noindent
Next, we take the expectation in (\ref{Dev:Eq-mu_n-2}). Remarking that $\left(A\phi\right)(x,x)\geq 0$, using Lemma \ref{Prelim:bound_alphageq1}.--\textit{(\ref{bound_alphageq1})}, since $\mu^n_0$ is deterministic, and since $M_{\alpha}\left(\mu^n_t\right) = \langle \mu^n_t(dx), \phi(x) \rangle$, we deduce:
\begin{eqnarray*}
 \mathbb E\left[M_{\alpha}\left(\mu^n_t\right)\right] & = & M_{\alpha}\left(\mu^n_0\right) + \frac{1}{2} \int_0^t \mathbb E\left[ \left\langle \mu^n_s(dx) \mu^n_s(dy), \left(A\phi\right)(x,y)K(x,y)\right\rangle \right] ds \\
&& \hspace{2.3cm} -\,\frac{1}{2n} \int_0^t \mathbb E\left[ \left\langle \mu^n_s (dx), \left(A\phi\right)(x,x)K(x,x)\right\rangle \right] ds,\\
&\leq & M_{\alpha}\left(\mu^n_0\right) + \frac{C_{\lambda,\alpha}}{2} \int_0^t \mathbb E\left[ \left\langle \mu^n_s(dx)  \mu^n_s(dy), x^{\alpha}y^{\lambda} + x^{\lambda}y^{\alpha}\right\rangle \right] ds\\
&\leq &M_{\alpha}\left(\mu^n_0\right) +  C_{\lambda,\alpha}\int_0^t \mathbb E\left[ M_{\lambda}\left(\mu^n_s\right) M_{\alpha}\left(\mu^n_s\right)\right] ds\\
&\leq & M_{\alpha}\left(\mu^n_0\right) +  C_{\lambda,\alpha}\, M_{\lambda}\left(\mu^n_0\right) \int_0^t \mathbb E\left[ M_{\alpha}\left(\mu^n_s\right)\right] ds,
\end{eqnarray*}
where we used the point \textit{(\ref{Prop:bounded_moments_neg})}. We conclude using the Gronwall lemma.
\end{proof}

Finally, we present the 
\begin{proof}[Proof of Lemma \ref{lemma_integrals}]
Assume that $\lambda\in(0,1]$ and recall (\ref{theta}). First, for \textit{(\ref{lemma:1})} and \textit{(\ref{lemma:2})}, by direct integration, we have
\begin{eqnarray*}
\int_0^{+\infty} x^{\lambda - 1}\theta_{(x)}^n dx & = &\frac{1}{\sqrt{n}}\int_0^{1}  x^{\lambda - 1} dx + \frac{1}{\sqrt{n}}\int_1^{+\infty}  x^{-\lambda - \varepsilon- 1} dx \\
& = &\frac{1}{\lambda\sqrt{n}} +  \frac{1}{(\lambda + \varepsilon)\sqrt{n}} \leq \frac{2}{\lambda \sqrt{n}},
\end{eqnarray*}
and
\begin{eqnarray*}
\int_0^{+\infty} x^{2\lambda - 1}\theta_{(x)}^n dx & = &\frac{1}{\sqrt{n}}\int_0^{1}  x^{2\lambda - 1} dx + \frac{1}{\sqrt{n}}\int_1^{+\infty}  x^{- \varepsilon- 1} dx \\
& = & \frac{1}{2\lambda\sqrt{n}} +  \frac{1}{\varepsilon\sqrt{n}} \leq \frac{(\lambda + \varepsilon)}{\lambda \varepsilon\sqrt{n}}.
\end{eqnarray*}
Next, for \textit{(\ref{lemma:3})} we have:\medskip
\begin{eqnarray*}
A_n(v,y) = v^{\lambda}\int_0^{+\infty}  \frac{x^{\lambda-1}}{\theta_{(x)}^n}\left(\mathds 1_{x <v\wedge y} + \mathds 1_{ v\vee y < x <v + y} \right)  dx & = & v^{\lambda} \int_{0}^{v\wedge y} \frac{x^{\lambda-1}}{\theta_{(x)}^n} dx +v^{\lambda}\int_{v\vee y}^{v + y} \frac{x^{\lambda-1}}{\theta_{(x)}^n} dx\\
& := & I_n(v,y) + J_n(v,y). 
\end{eqnarray*}
We have the following bounds: if $v \wedge y \leq 1$, then 
\begin{equation*}
I_n(v,y)  = v^{\lambda}\int_0^{v\wedge y} \sqrt{n}\, x^{\lambda-1}dx = \frac{\sqrt{n}}{\lambda}v^{\lambda}(v\wedge y)^{\lambda} \leq \frac{\sqrt{n}}{\lambda}v^{\lambda} y^{\lambda}.
\end{equation*}
Next, if $v \wedge y > 1$,
\begin{eqnarray*}
I_n(v,y) & = & v^{\lambda}\int_0^{1} \sqrt{n}\, x^{\lambda-1}dx  + v^{\lambda}\int_{1}^{v \wedge y} \sqrt{n}\, x^{3\lambda+\varepsilon-1} dx\\
& \leq & \sqrt{n}\, v^{\lambda} \left[ \frac{1}{\lambda} + \frac{(v \wedge y)^{3\lambda + \varepsilon}}{3\lambda + \varepsilon}\right]\\
& \leq & \frac{\sqrt{n}}{\lambda}\left[  v^{\lambda} + (v \wedge y)^{3\lambda + \varepsilon}v^{\lambda}\right]\\
& \leq & \frac{\sqrt{n}}{\lambda} \left[  v^{\lambda}y^{\lambda} + (v\wedge y)^{2\lambda} (v \vee y)^{2\lambda+\varepsilon} \mathds 1_{\lambda\in(0,1/2)} + (v\wedge y)(v\vee y)^{4\lambda+\varepsilon-1} \mathds 1_{\lambda\in[1/2,1]}\right].
\end{eqnarray*}
Thus, in any case
\begin{eqnarray*}
I_n(v,y) & = &  \frac{\sqrt{n}}{\lambda} \left[ v^{\lambda}y^{\lambda} + (v\wedge y)^{2\lambda} (v\vee y)^{2\lambda+\varepsilon} \mathds 1_{\lambda\in(0,1/2)} + (v\wedge y)(v\vee y)^{4\lambda+\varepsilon-1} \mathds 1_{\lambda\in[1/2,1]}\right].
\end{eqnarray*}
Next, since $x^{\lambda-1}$ and $\theta_{(x)}^n$ are non-increasing functions, according to the mean value theorem, we deduce that $ J_n(v,y) \leq v^{\lambda} \left( \frac{(v\vee y)^{\lambda-1}}{\theta_{(v+y)}^n}\right)(v\wedge y)$.\medskip

\noindent 
First, assume that $ v + y <1 $, then we get $J_n(v,y) \leq \sqrt{n}\,v^{\lambda} (v\vee y)^{\lambda-1} (v\wedge y) \leq \sqrt{n}\,v^{\lambda} y^{\lambda}$. \medskip

\noindent
Next, assume that $ v + y \geq 1 $, then 
\begin{eqnarray*}
J_n(v,y) & \leq & \sqrt{n}\,v^{\lambda} (v\vee y)^{\lambda-1}  (v + y)^{2\lambda+\varepsilon} (v\wedge y) \leq 2^{2\lambda+\varepsilon}\sqrt{n}\,v^{\lambda}  (v\wedge y) (v \vee y)^{3\lambda+\varepsilon-1} \\
& \leq & 2^{2\lambda+\varepsilon}\sqrt{n}\left[(v\wedge y)^{2\lambda}(v\vee y)^{2\lambda+\varepsilon}\mathds 1_{\lambda\in(0,1/2)} + (v\wedge y)(v\vee y)^{4\lambda+\varepsilon-1}\mathds 1_{\lambda\in[1/2,1]} \right].
\end{eqnarray*}
When $\lambda\in(0,1/2)$, we used $(v\wedge y) \leq (v\wedge y)^{2\lambda} (v\vee y)^{1-2\lambda}$ to deduce the bound $v^{\lambda}(v\wedge y)(v\vee y)^{3\lambda+\varepsilon-1} \leq v^{\lambda} (v\wedge y)^{2\lambda} (v\vee y)^{1-2\lambda} (v\vee y)^{3\lambda+\varepsilon-1} \leq (v\wedge y)^{2\lambda}(v\vee y)^{2\lambda+\varepsilon}$. \medskip

\noindent
Thus, in any case
\begin{equation*}
J_n(v,y) \leq \sqrt{n} v^{\lambda}y^{\lambda} + 2^{2\lambda+\varepsilon}\sqrt{n}\left[(v\wedge y)^{2\lambda}(v\vee y)^{2\lambda+\varepsilon}\mathds 1_{\lambda\in(0,1/2)} + (v\wedge y)(v\vee y)^{4\lambda+\varepsilon-1}\mathds 1_{\lambda\in[1/2,1]} \right].
\end{equation*}
Finally, we deduce the bound:
\begin{eqnarray*}
&& A_n(v,y) \leq \frac{2\sqrt{n}}{\lambda}  v^{\lambda}y^{\lambda} \\
&& \hspace{1cm} +\sqrt{n}\left(2^{2\lambda+\varepsilon} + \frac{1}{\lambda}\right)\left[(v\wedge y)^{2\lambda}(v\vee y)^{2\lambda+\varepsilon}\mathds 1_{\lambda\in(0,1/2)} + (v\wedge y)(v\vee y)^{4\lambda+\varepsilon-1}\mathds 1_{\lambda\in[1/2,1]} \right].
\end{eqnarray*}
This concludes the proof of Lemma \ref{lemma_integrals}.
\end{proof}
\end{appendix}

\bibliographystyle{plain.bst}        % (uses file alpha IEEEtranSA.bst"plain.bst")
\bibliography{Biblio}                % expects file "biblio.bib"

\begin{thebibliography}{10}

\bibitem{Aldous}
D.J. Aldous.
\newblock Deterministic and {S}tochastic {M}odels for {C}oalescence
  ({A}ggregation, {C}oagulation): {A} {R}eview of the {M}ean-{F}ield {T}heory
  of {P}robabilists.
\newblock {\em Bernoulli}, 5:3--48, 1999.

\bibitem{Rate}
M.~Deaconu, N.~Fournier, and E.~Tanr{\'e}.
\newblock Rate of {C}onvergence of a {S}tochastic {P}article {S}ystem for the
  {S}moluchowski {C}oagulation equation.
\newblock {\em Methodology and {C}omputing in {A}pplied {P}robability},
  (5):131--158, 2003.

\bibitem{Eibeck_Wagner}
A.~Eibeck and W.~Wagner.
\newblock {S}tochastic {P}article {A}pproximations for {Smoluchowski's}
  {C}oagulation {E}quation.
\newblock {\em Markov {P}rocesses and {R}elated {F}ields}, (9):103--130, 2003.

\bibitem{Dist-Coag}
N.~Fournier.
\newblock A distance for coagulation.
\newblock {\em Stoch{.} {P}roc{.} {A}ppl{.}}, 12(2):399--406, 2006.

\bibitem{Sto-Coal}
N.~Fournier.
\newblock On some stochastic coalescents.
\newblock {\em Proba{.} {T}heory {R}elated {F}ields}, 136(4):509--523, 2006.

\bibitem{Conv_ML}
N.~Fournier and J.~S. Giet.
\newblock Convergence of the {M}arcus-{L}ushnikov process.
\newblock {\em Methodology and {C}omputing in {A}pplied {P}robability},
  6:219--231, 2004.

\bibitem{Well-Pdnss}
N.~Fournier and Ph. Lauren\c{c}ot.
\newblock Well-possedness of {S}moluchowski{'}s {C}oagulation {E}quation for a
  {C}lass of {H}omogeneous {K}ernels.
\newblock {\em J{.} {F}unctional {A}nalysis}, 233(2):351--379, 2006.

\bibitem{Sto-Coal2}
N.~Fournier and E.~L{\"o}cherbach.
\newblock Stochastic coalescence with homogeneous-like interaction rates.
\newblock {\em Stoch{.} {P}roc{.} {A}ppl{.}}, 119:45--73, 2009.

\bibitem{Jacod}
J.~Jacod and A.~N. Shiryaev.
\newblock {\em Limit theorems for stochastics processes}.
\newblock Springer-{V}erlag, 2003.

\bibitem{Jeon}
I.~Jeon.
\newblock {E}xistence of {G}elling {S}olutions for
  {C}oagulation-{F}ragmentation {E}quations.
\newblock {\em Comm. {M}ath. {P}hys.}, 3:541--567, 1998.

\bibitem{Jourdain}
B.~Jourdain.
\newblock {N}onlinear {P}rocess {A}ssociated with the {D}iscrete {Smoluchowski}
  {C}oagulation-{F}ragmentation {E}quation.
\newblock {\em Markov {P}rocesses and {R}elated {F}ields}, (9):103--130, 2003.

\bibitem{Kolokoltsov}
V.~Kolokoltsov.
\newblock {T}he {C}entral {L}imit {T}heorem for the {S}moluchowski
  {C}oagulation {M}odel.
\newblock {\em arXiv:0708.0329v1 [math.PR]. {P}robability {T}heory and
  {R}elated {F}ields}, (1):87--153, 2010.

\bibitem{Lushnikov}
A.~Lushnikov.
\newblock {S}ome {N}ew {A}spects of {C}oagulation {T}heory.
\newblock {\em Izv. {A}kad. {N}auk {SSSR}, {S}er. {F}iz. {A}tmosfer. {I}
  {O}keana}, 14:738--743, 1978.

\bibitem{Marcus}
A.~Marcus.
\newblock {S}tochastic {C}oalescence.
\newblock {\em Technometrics}, 10:78--109, 1968.

\bibitem{Norris}
J.R. Norris.
\newblock Smoluchowski{'}s coagulation equation: uniqueness, non-uniqueness and
  hydrodynamic limit for the stochastic coalescent.
\newblock {\em Ann. Appl. Probab.}, 9(1):78--109, 1999.

\bibitem{Ziemer}
W.~P. Ziemer.
\newblock {\em Weakly differentiable functions, Graduate Texts in Mathematics}.
\newblock Springer-{V}erlag, 1989.

\end{thebibliography}
\end{document}